\documentclass[reqno]{amsart}

\usepackage{amsmath,amsthm,amsfonts,amssymb,young}
\usepackage{array, boldline, makecell, booktabs}
\usepackage{bm}
\usepackage{ytableau}
\usepackage{euscript, mathrsfs} 
\usepackage{tikz}
\usetikzlibrary{backgrounds}
\usepackage{mathtools}
\usepackage[colorlinks]{hyperref}
\usepackage{cleveref}
\usepackage[margin=1in]{geometry}
\numberwithin{equation}{section}    
\usepackage[all]{xy}
\usepackage{graphicx,import}
\usepackage{amscd}
\usepackage{color}
\usepackage{enumerate}
\usepackage{fourier}
\usepackage{cases}
\usepackage{cite}

\usepackage[T1]{fontenc}

\newtheorem{thm}{Theorem}[section]

\newtheorem{prop}[thm]{Proposition}
\newtheorem{cor}[thm]{Corollary}
\theoremstyle{definition}
\newtheorem{exam}[thm]{Example}
\newtheorem{defn}[thm]{Definition}
\newtheorem{conj}[thm]{Conjecture}

\newtheorem{remark}[thm]{Remark}

\newtheoremstyle{case}{}{}{}{}{}{:}{ }{}
\theoremstyle{case}
\newtheorem{case}{Case}

\newcommand{\cross}{$\mathbin{\tikz [x=1.4ex,y=1.4ex,line width=.2ex] \draw (0,0) -- (1.2,1.2) (0,1.2) -- (1.2,0);}$}%

  \definecolor{dredcolor}{rgb}{0.9,0.3,0.4}

\newcommand\LLT{\operatorname{LLT}}
\newcommand\area{\operatorname{area}}

\newcommand\SSS{\operatorname{S}}
\newcommand\SYT{\operatorname{SYT}}
\newcommand\SSYT{\operatorname{SSYT}}
\newcommand\asc{\operatorname{asc}}
\newcommand\inv{\operatorname{inv}}
\newcommand\Frob{\operatorname{Frob}}
\newcommand\inc{\operatorname{inc}}
\newcommand\maj{\operatorname{maj}}
\newcommand\stat{\operatorname{stat}}

\newcommand\comaj{\operatorname{comaj}}
\newcommand\wt{\operatorname{wt}}

\newcommand\Colo{\mathcal{C}}
\newcommand\Dyck{\operatorname{Dyck}}
\newcommand\Hess{\operatorname{Hess}}

\definecolor{cerulean}{rgb}{0.0, 0.48, 0.65}
\definecolor{dredcolor}{rgb}{0.9,0.3,0.4}

\title{$\alpha$-chromatic symmetric functions}

\author{Jim Haglund}
\address{
Department of Mathematics, University of Pennsylvania, Philadelphia, PA 19104, USA}
\email{jhaglund@math.upenn.edu}

\author{Jaeseong Oh}
\address{June E Huh Center for Mathematical Challenges, Korea Institute for Advanced Study, Seoul 02455, South Korea}
\email{jsoh@kias.re.kr}

\author{Meesue Yoo}
\address{
Department of Mathematics, Chungbuk National University, Cheongju 28644,
South Korea}
\email{meesueyoo@chungbuk.ac.kr}

\thanks{
The second author was supported by Korea Institute for Advanced Study (HP083401). 
The third author was supported by NRF grant RS-2024-00344076.
}



\begin{document}

\maketitle 

\begin{abstract}
In this paper, we introduce the \emph{$\alpha$-chromatic symmetric functions} $\chi^{(\alpha)}_\pi[X;q]$, extending Shareshian and Wachs' chromatic symmetric functions with an additional real parameter $\alpha$. We present positive combinatorial formulas with explicit interpretations. Notably, we show an explicit monomial expansion in terms of the $\alpha$-binomial basis and an expansion into certain chromatic symmetric functions in terms of the $\alpha$-falling factorial basis. Among various connections with other subjects, we highlight a significant link to $q$-rook theory, including a new solution to the $q$-hit problem posed by Garsia and Remmel in their 1986 paper introducing $q$-rook polynomials.
\end{abstract}

\section{Introduction}
\subsection{Overview}
In his seminal paper \cite{Sta95}, Stanley introduced the \emph{chromatic symmetric function} $\chi_G[X]$ as a symmetric function generalization of the chromatic polynomial for graphs. Subsequently, Shareshian and Wachs \cite{ShWa16} refined this, giving rise to the \emph{chromatic (quasi)symmetric function} $\chi_G[X;q]$ with an additional parameter $q$. Notably, Shareshian and Wachs showed that $\chi_{G(\pi)}[X;q]$ is symmetric when $G(\pi)$ is the graph  associated with a Dyck path $\pi$. For brevity, we adopt the notation $\chi_\pi[X;q] = \chi_{G(\pi)}[X;q]$. These symmetric functions have gathered substantial interest across diverse mathematical domains. For instance, various combinatorial formulas have been investigated \cite{Gas96, Atha2015, ShWa16}, and there is a wealth of research unveiling intriguing connections of chromatic symmetric functions with LLT polynomials \cite{CM2018}, Hessenberg varieties \cite{BC18, GP16}, Hecke algebras \cite{Hai93, CHBS16}, and rook theory \cite{AN21, CMP23, NT24}.

In this paper, we study a generalization of Shareshian and Wachs' chromatic symmetric functions. 
We introduce a new parameter $\alpha$ to define the \emph{$\alpha$-chromatic symmetric function} $\chi^{(\alpha)}_\pi[X;q]$ for a Dyck path $\pi$, for $\alpha \in \mathbb N$, as follows:
\begin{equation*}
    \chi^{(\alpha)}_\pi[X;q] := 
\chi_\pi\left[Q_{\alpha}X;q\right],
\end{equation*}
where $Q_{\alpha}:=\dfrac {q^\alpha-1}{q-1}$, 
and extend $\alpha$ to a real parameter using \emph{plethysm}. Notably, for a Dyck path of size $n$ (referred to here as an $n$-Dyck path), the coefficients of the $\alpha$-chromatic symmetric function $\chi^{(\alpha)}_\pi[X;q]$ lie within the $\mathbb{C}(q)$-span of $\left\{[\alpha]_q, [\alpha]_q^2, \dots, [\alpha]_q^n\right\}$, where $[\alpha]_q := \dfrac{q^\alpha-1}{q-1}$. We can express each coefficient using two alternative bases:
\begin{equation}\label{eq: two bases}
    \left\{\footnotesize{\begin{bmatrix}\alpha+k\\n\end{bmatrix}_q}\right\}_{k=0,1,\dots,n-1}\qquad\text{and}\qquad \left\{[\alpha]_q^{\underline{k}}\right\}_{k=1,2,\dots,n}.
\end{equation}
Here, \(
  [\alpha]_q^{\underline{k}} := [\alpha]_q[\alpha-1]_q\cdots[\alpha-k+1]_q
\)
is the $q$-falling factorial and $\footnotesize{\begin{bmatrix}\alpha+k\\n\end{bmatrix}_q} = \dfrac{[\alpha+k]_q^{\underline{n}}}{[n]_q!}$ denotes the $q$-binomial coefficient.
This paper presents various positive combinatorial formulas for the $\alpha$-chromatic symmetric functions in these bases or their specialization at $q=1$. Remarkably, Theorem~\ref{thm: alpha+k choose n basis monomial expansion} provides an explicit combinatorial formula for the monomial expansion, and Theorem~\ref{thm: falling factorial basis} offers an interpretation in terms of the falling factorial basis as a sum of certain chromatic symmetric functions. Corollary~\ref{Cor: Schur positivity alpha N} addresses the Schur expansion when $\alpha \in \mathbb{N}$, and Corollary~\ref{Cor: Schur positivity} discusses Schur positivity in terms of two binomial bases.

\subsection{Motivations}
Our results and conjectures intersect with several significant topics, including the \emph{LLT polynomials}. Introduced by Lascoux, Leclerc, and Thibon \cite{LLT97}, the LLT polynomials represent a class of symmetric functions that can be viewed as a $q$-analogue of a product of skew Schur functions. These polynomials are indexed by tuples of skew partitions, and when the tuple consists of single cells, they are called the \emph{unicellular} LLT polynomials. Although Grojnowski and Haiman \cite{GH07} established the Schur positivity of the LLT polynomials, finding a combinatorial formula for the Schur coefficients remains challenging. Interestingly, the Schur coefficients of the $\alpha$-chromatic symmetric functions are related to those of the unicellular LLT polynomials (Proposition~\ref{prop: Schur coefficient of LLT}). Thus, studying the Schur expansion of the $\alpha$-chromatic symmetric functions in terms of the two bases in \eqref{eq: two bases} could shed light on the remarkable problem of finding a combinatorial interpretation for the Schur coefficients of the LLT polynomials.

The second motivation for studying the Schur expansion of the $\alpha$-chromatic symmetric functions arises from analogous phenomena observed with the \emph{(integral form) Jack polynomials}. Jack polynomials, which include an additional parameter $\alpha$, form an important family of symmetric functions with applications in mathematical physics, representation theory, and combinatorics. The process of obtaining the $\alpha$-chromatic symmetric functions from the unicellular LLT polynomials bears a resemblance to the derivation of the Jack polynomials from the modified Macdonald polynomials. While a combinatorial formula for the monomial expansion of the Jack polynomials is known \cite{KS97, HHL05}, the Schur expansion has been less explored since the Jack polynomials are not Schur positive in general. Nevertheless, Alexandersson, Haglund, and Wang \cite{AHW21} conjectured that after certain rescaling, the Jack polynomials expand positively in terms of 
\[ 
    \displaystyle \left\{\footnotesize{\binom{\alpha+k}{n}} s_\lambda\right\}_{\substack{0\le k \le n-1\\ \lambda\vdash n}}, \qquad \text{or} \qquad \left\{(\alpha)^{\underline{k}} s_\lambda\right\}_{\substack{1\le k \le n \\ \lambda \vdash n}}.
\]
Given the similarity between the Jack polynomials and the $\alpha$-chromatic symmetric functions, we hope our findings on the $\alpha$-chromatic symmetric functions will provide new insights on this problem concerning the Schur expansion of the Jack polynomials.

Finally, we explore the connection between the theory of chromatic symmetric functions and rook theory. Previous research has linked rook theory with chromatic symmetric functions, particularly in the `abelian' case, as noted in \cite{SS93, AN21, CMP23}. In this paper, we use rook theory to explore the combinatorial properties of the $\alpha$-chromatic symmetric functions. Briefly, given a sequence of nonnegative integers $c_1,\ldots, c_n$ satisfying $0 \le c_1 \le c_2 \le \cdots \le c_n \le n$, we let $B$ denote the Ferrers board whose $i$th column is of height $c_i$, $1 \le i \le n$. (See \cite[Ch. 3]{StaEC} for background on rook theory). We define $\text{PR}(B)= \prod_{i=1}^n [\alpha+c_i-i+1]_q$, and call any product of $n$ many $q$-integers that can be written as $\text{PR}(B)$ for some Ferrers board $B$ a \emph{rook product}. Garsia and Remmel showed that for any Ferrers board $B$,
\begin{align*}
  \text{PR}(B) = \sum_{k=0}^n {\footnotesize{\begin{bmatrix}\alpha+k\\n\end{bmatrix}_q} h_k(B;q)}
    = \sum_{k=1}^{n} [\alpha]_q^{\underline{k}} r_{n-k} (B;q)
\end{align*}
where \( r_k (B;q)\in \mathbb{N}[q] \) is the \( q \)-rook polynomial, $h_k(B;q) \in \mathbb{N}[q]$, and $h_k(B;1)$ equals the $k$th hit number of $B$. They posed the problem of finding a statistic $\text{qstat}$ defined on rook placements $C$ of $n$ nonattacking rooks (no two in the same row or column) in the $n \times n$ square grid containing $B$, with exactly $k$ rooks on $B$ (i.e., permutations in $\mathfrak{S}_n$ that `hit' $k$ many forbidden positions encoded by the cells of $B$), such that $h_k(B;q) = \sum_{C} q^{\text{qstat}(C,B)}$. This problem was solved independently by Dworkin and the first author \cite{Hag98a, Dwo98}. In this paper, we provide a new solution to this problem, where the $q$-statistic we obtain has a natural interpretation in terms of colorings of incomparability graphs.

\subsection{Outline}
The structure of this paper is as follows. Section~\ref{Sec:Background} contains background material on symmetric functions and rook theory that we use throughout the paper. For a thorough treatment of common notation and basic properties of the ring of symmetric functions, we refer the reader to \cite[Ch. 1]{Mac1995} or \cite[Ch. 7]{Stanley2}. We also 
compute $\chi_\pi\left[Q_{\alpha};q\right]$
using plethysm and the technique of \emph{superization} developed in 
\cite{HHLRU, HHL05} (see also \cite[Ch. 6]{Hag2008}).

Section~\ref{Sec: binomial alpha+k choose n basis expansion} provides the first main theorem (Theorem~\ref{thm: alpha+k choose n basis monomial expansion}), giving a combinatorial formula for coefficients of $\chi _{\pi} ^{(\alpha)}[X;q]$ in terms of the
\[
\left\{ \footnotesize{\begin{bmatrix}\alpha+k\\n\end{bmatrix}_q}  m_\lambda\right\}_{\substack{1\le k \le n\\ \lambda\vdash n}}
\]
basis. We use the superization technique and a bijective argument to prove it.

In Section~\ref{Sec: bracket alpha choose k basis}, we first introduce a novel \emph{$XY$-technique}. This $XY$-technique leads to various positivity results. In particular, we prove the second main result (Theorem~\ref{thm: falling factorial basis}), expressing the falling factorial coefficient of $\alpha$-chromatic symmetric functions as a sum of certain chromatic symmetric functions.

Section~\ref{Sec: Schur positivity} concerns Schur positivity of the $\alpha$-chromatic symmetric functions in terms of two types of $\alpha$-binomial bases. We also discuss certain symmetry between Schur coefficients of the $\alpha$-chromatic symmetric functions.

In Section~\ref{Sec:qHitpoly}, we show that for any Dyck path $\pi$, $\chi_\pi[Q_{\alpha};q]$ is a rook product and use this fact to obtain various expansions of $\chi_\pi[Q_{\alpha};q]$ into the $\footnotesize{\begin{bmatrix}\alpha+k\\n\end{bmatrix}_q}$ and $\footnotesize{\begin{bmatrix}\alpha\\k\end{bmatrix}_q}$ bases. We show how these results, when combined with the known expansion of $\chi_\pi[X;q]$ into Gessel's fundamental quasisymmetric functions, give a new formula for the hit polynomial $h_k(B;q)$, namely, a new solution to the Garsia--Remmel $q$-hit polynomial problem.

Finally, we end this paper with Section~\ref{Sec: further remarks} which contains remarks about motivating questions and other related problems such as geometric interpretations.

\section{Background}\label{Sec:Background}

\subsection{Symmetric functions}\label{subsec: symmetric functions}

A \emph{partition} $\lambda=( \lambda_1, \lambda_2,\dots, \lambda_{\ell})$ of \( n \) is a nonincreasing sequence of positive integers satisfying \( \sum_i \lambda_i =n\). We use \( \lambda\vdash n \) to denote that $\lambda$ is a partition of $n$. 
Each \( \lambda_i \) is called a \emph{part} of \( \lambda \) and the number of parts in \( \lambda \) is called
the \emph{length} of \( \lambda \), denoted by $\ell(\lambda)$. We identify a partition \( \lambda \) with its \emph{(Young) diagram} which is
the set of cells \( \{(i, j)\, :\, 1\le i \le \ell(\lambda), \, 1\le j\le \lambda_i \} \). We use the \emph{English notation}
to draw the diagram, putting left justified $\lambda_i$ cells in each row from top to bottom.

Let \( \Lambda_F \) denote the \( F \)-algebra of symmetric functions with coefficients in \( F=\mathbb{Q}(q,t) \).
The elements in \( \Lambda_F \) may be viewed as formal power series over \( \mathbb{Q}(q,t) \) in infinitely many
variables \( X=(x_1, x_2,\dots) \) with finite degrees that are invariant under permutations of variables.
In this paper, we often replace \( t\) by \( q^\alpha \), for some \( \alpha \).

For \( \lambda\vdash n \), the \emph{monomial symmetric functions} \( m_{\lambda} \)
are defined by
\(
  m_{\lambda}[X] = \sum_{\nu}X^{\nu},
\)
where the sum is over weak compositions \( \nu \) such that the nonincreasing rearrangements of the parts of \( \nu \) equals
\( \lambda \), and \( X^\nu=x_1^{\nu_1}x_2^{\nu_2}\cdots x_n^{\nu_n} \), for \( \nu=(\nu_1,\dotsm \nu_n) \). 
The \emph{elementary symmetric functions} \( e_\lambda \) are defined by \( e_\lambda =\prod_i e_{\lambda_i}\),
where \( e_0=1 \) and \( e_n=\sum_{ i_1<\cdots <i_n} x_{i_1}\cdots x_{i_n}\), for \( n\ge 1\).
The \emph{complete homogeneous symmetric functions} \( h_\lambda \) are defined by \( h_\lambda = \prod_i h_{\lambda_i} \),
where \( h_0 =1 \) and \( h_n=\sum_{ i_1\le \cdots \le i_n} x_{i_1}\cdots x_{i_n}\), for \( n\ge 1\).
The \emph{power sum symmetric functions} \( p_\lambda \) are defined by \( p_\lambda =\prod_i p_{\lambda_i} \),
where \( p_0 =1  \) and \( p_n = \sum_i x_i ^n \) for \( n\ge 1 \).

A \emph{semistandard Young tableau} of shape \( \lambda \) is a filling of \( \lambda \)
with letters in \( \mathbb{Z}_{>0} \) such that the entries are weakly increasing across each row and strictly increasing up 
each column. For a semistandard Young tableau \( T \), we define the \emph{weight} \( \wt(T) \) of \( T \) to be the sequence
 \(( \wt_1(T), \wt_2(T),\dots) \), where \( \wt_i (T) \) is the multiplicity of \( i \) in \( T \). Especially when the weight
is \( (1,1,1,\dots,1) \), \( T \) is called \emph{standard}. Let \( \SSYT(\lambda) \) (\( \SYT(\lambda) \), resp.) denote the set of semistandard Young tableaux (standard Young tableaux, resp.)
of shape \( \lambda \). Then the \emph{Schur functions} are defined by
\[
s_\lambda[X]=\sum_{T\in \SSYT(\lambda)}X^{\wt(T)}.
\]

We let \( \omega \) be the involution on symmetric functions such that \( \omega(p_\lambda [X]) =(-1)^{|\lambda|-\ell(\lambda)}p_\lambda [X]\). It sends elementary symmetric functions to homogeneous symmetric functions, \( \omega(e_n [X]) =h_n [X]\), and Schur functions to Schur functions, \( \omega(s_\lambda [X])=s_{\lambda'}[X] \). Here $\lambda'$ is the conjugate partition of $\lambda$.


Next, we briefly review plethystic notation. 
Let \( E=E(t_1, t_2, \dots) \) be a formal Laurent series with rational coefficients in \( t_1, t_2,\dots \). The
\emph{plethystic substitution} \( p_k [E] \) is defined by replacing each \( t_i \) in \( E \) by \( t_i ^k \), i.e.,
\( p_k[E] = E\left(t_1 ^k, t_2 ^k, \dots\right)\) and
\[
  p_\lambda [E] = \prod_{i=1}^{\ell(\lambda)} p_{\lambda_i}[E].
\]
Here, we summarize some of the basic properties of plethystic substitution.
\begin{enumerate} 
\item If \( X=(x_1,x_2,\ldots)\), then 
$\displaystyle p_k[X]=\sum_{i\ge 1}x_i ^k=p_k (x_1,x_2,\dots)$.
\item For a real parameter $z$, $p_k [zX]=z^k p_k[X]$.
\item $p_k[(1-t)X]=\sum_{i}x_i ^k (1-t^k)=(1-t^k)p_k[X]$.
\item $\displaystyle p_k \left[ \frac{X}{1-q}\right]=\sum_{i}\frac{x_i ^k}{1-q^k}=\frac{1}{1-q^k}p_k[X]$.
\end{enumerate}
Then for an arbitrary symmetric function \( f \), we define \( f[E] \) by \( \sum_{\lambda}c_\lambda p_\lambda [E] \)
if \( f=\sum_{\lambda}c_\lambda p_\lambda \). Due to the above property (1), it easily follows that for any \( f\in \Lambda_F \), \( f[X]=f(x_1,x_2,\dots) \). For this reason, it is a common convention
in plethystic expression that \( X \) stands for the summation of variables \( x_1+x_2+\cdots \). In a similar vein, for $X=x_1+x_2+\cdots, Y=y_1+y_2+\cdots$, we have $f[X+Y]=f(x_1,x_2,\dots,y_1,y_2,\dots)$ and $f[XY]=f(x_1y_1,x_1y_2,\dots,x_2y_1,x_2y_2,\dots,)$.

If we let \( Z=(-x_1,-x_2,\dots) \), then \( p_k(Z)=\sum_i (-1)^k x_i ^k \) which is different from
\( p_k [-X]=\sum_i( -x_i ^k) \). Thus we introduce a symbol \( \epsilon \) to denote the negation of variables inside plethystic brackets, i.e., \( p_k[\epsilon X]=\sum_i (-1)^k x_i ^k \).
Then note that for any symmetric function \( f \), we have \( \omega(f(X))=f[-\epsilon X] \).
For a more comprehensive account, see \cite{Hag2008, Mac1995}.

\subsection{Combinatorics}\label{subsec: Combinatorics}

For a fixed $n$, an $n$-\emph{Dyck path} is a lattice path from $(0,0)$ to $(n,n)$ consisting of the up-steps $(0,1)$ and right-steps $(1,0)$ which stays above the main diagonal. There is a natural way to associate a Hessenberg function $h(\pi)$, a natural unit interval order $P(\pi)$, and a graph $G(\pi)$ for a Dyck path $\pi$. We outline those correspondences. 

Throughout this paper, we set $[n]=\{ 1,2,\dots, n\}$. A \emph{Hessenberg function} is a nondecreasing function $h:[ n] \rightarrow  [n]$ such that $h(i)\ge i$ for all $1\le i \le n$. For a $n$-Dyck path $\pi$, if we let $h(\pi)$ be a function whose value at $i$ is the $i$-th column height of $\pi$, then $h(\pi)$ is a Hessenberg function. Therefore, a Dyck path determines a Hessenberg function and vice versa.

A \emph{natural unit interval order} is a poset arising from an arrangement of unit intervals.
It is known by Wine and Freund \cite{WF1957} 
that unit interval orders are counted by Catalan numbers.  Furthermore, two different bijections to associate a unit interval order to a Dyck path are given by Skandera and Reed \cite{SR2003}, and Guay-Paquet, Morales and Rowland \cite{GMR2014}. So, we can represent natural unit interval orders by Dyck paths. 

Let us describe the correspondence. For a Dyck path $\pi$, $P=P(\pi)$ is a poset on $[n]$ whose order relation is defined by \( i<_{P(\pi)} j \) if the cell $(i,j)$ lies above $\pi$.
In this case, we also denote the order relation by $i<_\pi j$ for \( i<_{P({\bf m})} j \). In addition, we say $(i,j)$ is an \emph{attacking pair} of $\pi$, if $1\le i<j\le n$ and $i\nless_\pi j$. 

Note that a poset $P$ is called \emph{$(r + s)$-free} if $P$ does not contain an induced subposet isomorphic to the direct sum of an $r$-element chain and an $s$-element chain. 
It is well-known that a poset $P$ is a natural unit interval order 
if and only if it is $(2+2)$-free and $(3+1)$-free.

Given a poset $P$, the \emph{incomparability graph}  $\inc(P)$ is a graph whose vertex set is the elements of $P$
with edges connecting pairs of incomparable elements in $P$. Notice that for a natural unit interval order $P(\pi)$ associated with a Dyck path $\pi$, the corresponding incomparability graph is encoded in the cells between the Dyck path and the main diagonal. We denote this graph by $G(\pi)$. 

When referring to a Dyck path $\pi$, we will implicitly encompass the associated objects discussed above. In other words, depending on context $\pi$ may represent not only the Dyck path itself but also its corresponding Hessenberg function $h(\pi)$, the natural unit interval order $P(\pi)$, and the associated incomparability graph $G(\pi)$.

\begin{exam}\label{exa:1}
  Given a Dyck path \( \pi =\text{NNENNEENEE} \), where N denotes the up step and E denotes the east step, starting from the south left corner, the corresponding Hessenberg function \( h(\pi) \) has
  the function value \( (2, 4, 4, 5, 5) \), where \( i \) th entry \( h(i) \) is the height of the \( i \) th column of \( \pi \).

  The corresponding natural unit interval order \( P(\pi)  \) is the poset with the relations
  \[
    1<_\pi 3,\, 1<_\pi 4, \, 1<_\pi 5,\, 2<_\pi 5,\, 3<_\pi 5.
  \]
For the poset \( P(\pi) \), the incomparability graph \( G(\pi) \) is the graph with the vertex set \( \{ 1, 2, 3, 4, 5\}\) and the edge set  \( \{ \{ 2,1\}, \{3,2\}, \{ 4,2\}, \{4,3\}, \{ 5,4\}\} \).
See Figure \ref{fig:exofpi}.
\begin{figure}[!ht]
\begin{tikzpicture}[scale=.6]
 \foreach \i in {0,...,5} {
 \draw[color=gray] (0,\i)--(5,\i);
  \draw[color=gray] (\i,0)--(\i,5);
  }
  \foreach \i in {1,...,5}{
\node () at (-.5, 5-\i+.5) {$\i$};
\node () at (5-\i+.5, 5.5) {$\i$};
}
\draw [dashed] (0,0)--(5,5);
\draw [very thick] (0,0)--(0,2)--(1,2)--(1,4)--(3,4)--(3,5)--(5,5)--(5,0)--(0,0)--cycle;
\node () at (-2, 2.5) {\( \pi = \)};
\end{tikzpicture}\\
\begin{tikzpicture}[scale=1]
  \draw [thick](.2,.2)--(.85, .85)
  (1.15, 0.85)--(1.85, .15)
  (2.15,-.15)--(2.85, -.85)
  (3.15,-.85)--(3.85, -.15);
  \node () at (0,0) {\( \bf{2} \)};
  \node () at (1,1) {\( \bf{5} \)};
  \node () at (2,0) {\( \bf{3} \)};
  \node () at (3,-1) {\( \bf{1} \)};
  \node () at (4,0) {\( \bf{4} \)};
  \node () at (-1, 0) {\( P(\pi)\,: \)};
  \node () at (5.7, 0) {\( G(\pi)\,: \)};
  \filldraw (6.5,0) circle (2.5pt)
  (7.5,0) circle (2.5pt)
  (8.5,0) circle (2.5pt)
  (9.5,0) circle (2.5pt)
  (10.5,0) circle (2.5pt);
  \draw [thick] (6.5,0)--(10.5, 0);
  \draw [thick]  (7.5,0) to[out=70,in=110] (9.5,0);
  \node () at (6.5, -.3) {\( 1 \)};
  \node () at (7.5, -.3) {\( 2 \)};
  \node () at (8.5, -.3) {\( 3 \)};
  \node () at (9.5, -.3) {\( 4 \)};
  \node () at (10.5, -.3) {\( 5 \)};
  \end{tikzpicture}
\caption{A Dyck path \( \pi \) and corresponding \( P(\pi) \) and \( G(\pi) \).}\label{fig:exofpi}
\end{figure}
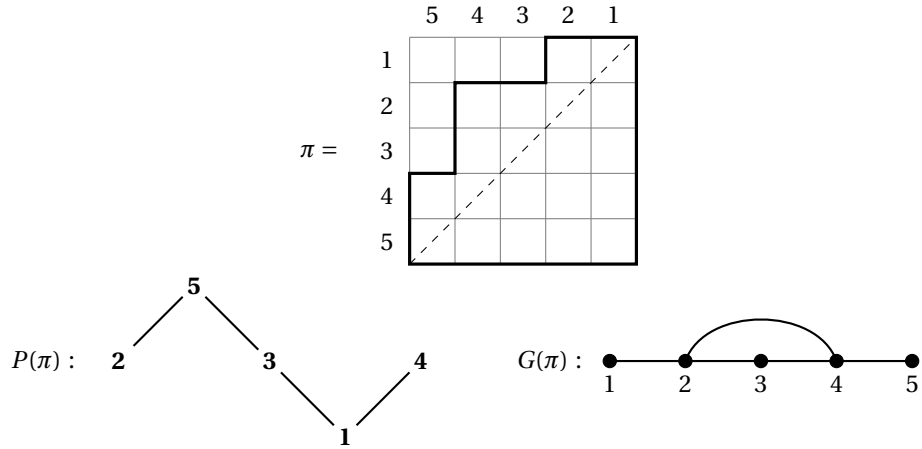
\end{exam}

\subsection{Chromatic symmetric functions and unicellular LLT polynomials}
\label{subsec:CSFandLLT}

For an $n$-Dyck path $\pi$ and a word $w\in \mathbb{Z}_{>0}^n$ of length $n$, we define the \emph{inversion} statistic as 
$\inv_{\pi} (w)= |\{ (i,j) ~:~ i<j,\, i\nless_{\pi} j,\,\text{ and } w_i > w_j\}|$.

The \emph{chromatic symmetric function} indexed by  $\pi$ is defined by 
\begin{equation}\label{def:chromsym}
  \chi_{\pi} [X;q]=\sum_{\substack{w\in \mathbb{Z}_{>0}^n\\ \text{proper}}}q^{\inv_{\pi}(w)}x^w,
\end{equation}
where the sum is over `proper' colorings of $G(\pi)$, that is, over words of length $n$ such that $w_i\ne w_j$ for $i\nless_{\pi} j$. We can think of the terms in this sum as placements of the words $w_1w_2\cdots w_n$ in the main diagonal cells with $w_i$ in the cell  $(i,i)$ in $\pi$ in Figure \ref{fig:exofpi}, where inversions correspond to cells $(i,j)$, with the number $w_i$ in the $i$th row from the top being larger than the number $w_j$ in the $j$th column from the right. We note that since $ \chi_{\pi} [X;q]$ is symmetric, we can reverse the variable set without changing the function.  Equivalently, if we define
$\asc_{\pi} (w)= |\{ (i,j) ~:~ i<j,\, i\nless_{\pi} j,\,\text{ and } w_i < w_j\}|$ for $w\in \mathbb{Z}_{>0}^n$ of length $n$, then we have 
\begin{equation}\label{def:chromsymasc}
  \chi_{\pi} [X;q]=\sum_{\substack{w \in \mathbb{Z}_{>0}^n\\ \text{proper}}}q^{\asc_{\pi}(w)}x^w.
\end{equation}
Note that the function $\chi_\pi[X;q]$ can be defined as a form of \eqref{def:chromsymasc} for any labeled graphs and the function is not in general symmetric but quasisymmetric. However, it is proved by Shareshian and Wachs \cite[Theorem 4.5]{ShWa16} that the chromatic quasisymmetric function is symmetric for incomparability graphs of natural unit interval orders. We will only consider the symmetric case in this paper, and take \eqref{def:chromsym} as the definition of $\chi _{\pi}$.   

In definition (\ref{def:chromsym}) for the chromatic symmetric function $\chi_{\pi}$, if we remove the condition for $w$ to be proper ($w_i\ne w_j$ for $i\nless_{\pi} j$), then it becomes the \emph{unicellular LLT polynomial} indexed by $\pi$, namely 
\begin{equation}\label{def:unicellularLLT}
\LLT_{\pi}[X;q]=\sum_{w\in \mathbb{Z}_{>0}^n}q^{\inv_{\pi}(w)}x^w.
\end{equation}
The original definition of the LLT polynomials given by Lascoux, Leclerc, and Thibon \cite{LLT97}
uses \emph{cospin} statistic defined on ribbon tableaux. Later, Bylund and Haiman found
an equivalent definition using \emph{dinv} statistic defined over \( k \)-tuple of semistandard
Young tableaux of skew shapes.
In the case when each skew shape consists of only one cell, the LLT polynomials are called unicellular.
In this case, the diagram of a tuple of single cells corresponds to a Dyck path \cite{Hag2008,CM2018} and the definition can be rewritten as \eqref{def:unicellularLLT}.

Carlsson and Mellit \cite{CM2018} utilized an explicit relationship between the unicellular LLT polynomials and the chromatic symmetric functions via a plethystic substitution.
\begin{prop} \label{prop: plethystic relation between X and LLT Carlsson Mellit} For an $n$-Dyck path $\pi$, we have
\begin{equation}\label{eqn:CMrelation}
    \chi_\pi[X;q] = \dfrac{\LLT_{\pi}[(q-1)X;q]}{(q-1)^n}.
\end{equation}
\end{prop}
By replacing \( (q - 1) \) in the numerator of \eqref{eqn:CMrelation} within the plethysm with \( (t - 1) \), and subsequently substituting \( t = q^\alpha \), we define the \( \alpha \)-chromatic symmetric functions as follows.
\begin{defn}\label{def:alpha_chromsym}
Given an \( n \)-Dyck path $\pi$ and $\alpha\in \mathbb{R}$, we define the \emph{$\alpha$-chromatic symmetric function} by
\begin{equation}\label{eqn:alpha_chromsymdef}
   \chi^{(\alpha)}_\pi[X;q] := \chi_\pi\left[Q_\alpha X;q\right] =\frac{ \LLT_{\pi}[(q^\alpha-1)X;q]}{(q-1)^n}.
\end{equation}
\end{defn}

\subsection{Superization}\label{subsec: superization}
To deal with the plethystic substitution used in \eqref{eqn:alpha_chromsymdef},
we review the standard technique of \emph{superization} following \cite{HHLRU, HHL05} and \cite[pp. 99-101]{Hag2008}.   Let 
$\mathcal A_{+} = \{1,2,\ldots ,n\}$ and $\mathcal A _{-} = \{\overline  1, \overline 2,\ldots ,\overline n\}$ denote positive and negative alphabets,  respectively.  Then for any fixed ordering of the 
alphabets the following holds. 

\begin{prop}\label{prop:superization}
    Let $f$ be a symmetric function of homogeneous degree $n$, written in terms of quasisymmetric functions as 
    \[
        f[X] = \sum_{\sigma \in \mathfrak{S}_n} c_{\sigma}  F_{n,\text{Des}(\sigma ^{-1})}[X],
    \]
    where the  $c_{\sigma}$ are constants independent of $X$.
    Then $\tilde{f}[X,Y]=\omega_Y f[X+Y]$ (called the {\it superization} of $f$) has the expansion 
    \[
        \tilde{f}[X,Y]= \sum_{\sigma \in \mathfrak{S}_n} c_\sigma \tilde{F}_{n,\text{Des}(\sigma ^{-1})}[X,Y].
    \]
    Here for $D \subseteq [n-1]$, $F_{n,D}[X]$ is Gessel's fundamental quasisymmetric function corresponding to $D$ and  ${\tilde F}_{n,D}[X,Y]$ is its superization, defined as
    \begin{align*}
  F_{n,D}[X] &= \sum _{1 \le a_1\le a_2 \le \cdots \le a_n \atop {a_i=a_{i+1} \implies i \notin D} } \prod_{i=1}^n x_{a_i}, \qquad  \text{and} \\
    {\tilde F}_{n,D}[X,Y] &= \sum _{1\le a_1\le a_2 \le \cdots \le a_n \atop {a_i=a_{i+1} \in \mathcal A_{+}  \implies i \notin D  \atop {                                                                              
    a_i=a_{i+1} \in \mathcal A_{-}  \implies i \in D }}} 
    \prod_{i: a_i \in \mathcal A_{+} } x_{a_i}      \prod_{i: a_i \in \mathcal A_{-} } y_{|a_i|},
     \end{align*}
     with $|\overline i|=i$ for $1\le i \le n$.
\end{prop}
\begin{defn}. Given a word $w \in \mathbb Z _{>0}^n$, we let $\text{stan}(w)$ denote the {\it standardization} of $w$, which is the unique permutation $\beta \in \mathfrak{S}_n$
satisfying that if $w_i<w_j$ then $\beta _i <\beta _j$, and if $w_i=w_j$ for $i<j$, then $\beta _i < \beta _j$.  For example, $\text{stan}(323662) = 314562$.  More generally, let \( \mathcal{A}_{\pm} = \mathcal{A}_{+} \cup \mathcal{A}_{-} \). Given a \emph{superword} \( \tilde{w} \in \mathcal{A}_{\pm}^n \), we define \( \text{stan}(\tilde{w}) \) to be the permutation \( \beta \), as defined above, with the additional convention that if \( \tilde{w}_i = \tilde{w}_j \in \mathcal{A}_{-} \) and \( i < j \), then \( \beta_i > \beta_j \).
So, equal negative letters standardize in the opposite way of equal positive letters.  For example, assuming the
ordering $1 < \overline 1 < 2 <  \overline 2 <\cdots <n < \overline n$, $\text{stan}(2 \, \overline 2  \, \overline 2 \,  \overline 1 \,  \overline 1 \, 1  \, 1 \, 2 \, \overline 1)  = 6985 41273$.
\end{defn}

Let $X_n = \{x_1,\ldots ,x_n\}$ and $Y_n=\{y_1,\ldots ,y_n\}$ be two finite sets of variables.  It is a fairly straightforward exercise to show that for any $\sigma \in \mathfrak{S}_n$,
\begin{align}
\label{Standard}
F_{n,\text{Des}(\sigma ^{-1})} [X_n] &= \sum_{w \in \mathcal A_{+}^n \atop \text{stan}(w) = \sigma}  x^{w}, \qquad \text{and} \\\label{Standard2}
 \tilde F_{n,\text{Des}(\sigma ^{-1})}[X_n,Y_n] &= \sum_ {\tilde w \in \mathcal A_{\pm}^n \atop \text{stan}(\tilde w) = \sigma } 
 \prod_{i: w_i \in \mathcal A_{+}} x_{\tilde w_i}
 \prod_{i: \tilde w_i \in \mathcal A_{-}} y_{|\tilde w_i|}.
\end{align}

Let $G[X;q]$ be a symmetric function of homogeneous degree $n$,  depending on a parameter $q$, satisfying
 \begin{align*}
G[X;q] = \sum_ { w \in \mathcal A_{+} ^n }
 q^ { \stat(w) }  x^{w}.
 \end{align*}
Proposition \ref{prop:superization}, \eqref{Standard}, and \eqref{Standard2} imply that if 
 $\stat$ is invariant under standardization, then
 \begin{align}
 \label{TStandard}
 \omega ^Y G[X+Y; q] = \sum _{\tilde w \in \mathcal A_{\pm} ^n} q^{ \stat(\tilde w)}
 \prod_{i: \tilde w _i \in \mathcal A_{+} } x_{a_i}      \prod_{i: \tilde w _i \in \mathcal A_{-} } y_{|a_i|}.
\end{align}
Here $\omega ^Y$ sends $s_{\lambda}(Y)$ to $s_{\lambda ^{\prime}}(Y)$, and fixes the $X$ variables.
Note that since \( \omega s_\lambda(Y)=s_\lambda[-\epsilon Y] \), we can denote the left hand side of
\eqref{TStandard} as 
\begin{equation}\label{eqn:superization}
\omega ^Y G[X+Y; q]=G[X-\epsilon Y;q]. 
\end{equation}

Given an $n$-Dyck path $\pi$, for $1\le i \le n$, let  $b_i(\pi)$ and $a_i(\pi)$ denote the number of cells below $\pi$ and strictly above the diagonal $y=x$, 
in the $i$th column (row, respectively), from the right (top, respectively).  For example, for the path on the left in Figure \ref{XYfig} of Section~\ref{Sec:XYTechnique}, we have
$(b_1,\ldots ,b_n)= (0,1,1,0,1,2,1,1,2)$ and $(a_1,\ldots ,a_n)= (1,1,0,2,1,1,2,1,0)$.  Also, let $\area(\pi) = \sum_i a_i = \sum_i b_i$.
\begin{prop}\label{prop: q-chromatic}
\label{RProd}
For an $n$-Dyck path $\pi$ and a real parameter $\alpha$,
\begin{align}
\label{QQ}
\chi _{\pi} [Q_{\alpha};q]  = q^{\area(\pi)}\prod_{i=1}^n [\alpha -a_i(\pi)]_q = q^{\area(\pi)}\prod_{i=1}^n [\alpha -b_i(\pi)]_q.
\end{align}
\end{prop}
\begin{proof}
  By \eqref{eqn:CMrelation}, we have $(q-1)^n \chi_{\pi}[Q_\alpha;q] = \LLT_{\pi} [q^{\alpha}-1;q]$. To deal with the plethystic substitution of $q^{\alpha}-1$ into $\LLT_{\pi}$,
  we use the \eqref{eqn:superization}.  Fix the ordering $1 < \overline 1 < 2 <  \overline 2 <\cdots <n < \overline n$  in 
  $\mathcal A_{\pm}$, and for a superword
  $\tilde w \in \mathcal A _{\pm}^n$, let \( \inv_\pi (\tilde{w})=\inv_\pi (\text{stan}(\tilde{w})) \).

By using the definition of \( \LLT_{\pi} [X;q] \) in (\ref{def:unicellularLLT}), since $\inv$ is invariant under standardization,  
 we try to express $\LLT_{\pi}[q^{\alpha}-1;q]$ as $\LLT_{\pi}[X-\epsilon Y;q]$ for some choice of $X,Y$ so that we can write  $\LLT_{\pi}[q^{\alpha}-1;q]$
as a weighted sum over superwords. That is, we want to 
choose $X,Y$ so that $X-\epsilon Y = q^{\alpha}-1$.  Since $\epsilon ^2 =1$, one way to do this is to let $X=q^{\alpha}$ (the set of variables with one element with 
weight $q^{\alpha}$), and $Y= \epsilon$ (the set of variables with just one element of weight $-1$).  With this choice, by (\ref{TStandard}) we have
\begin{align}
\label{vector}
(q-1)^n \chi_{\pi}^{(\alpha)}[1;q]  = \sum_{\tilde w \in \{1,\bar{1}\}^n}   q^{\inv_{\pi}(\tilde{w})} q^{\alpha \text{pos}(\tilde{w})}(-1)^{\text{neg}(\tilde{w})},
\end{align}
where the sum is over all superwords $\tilde{w}$ of length $n$ 
consisting of $1$'s and $\overline 1$'s, with $\text{pos}(\tilde{w})$ denoting the number of $1$'s and $\text{neg}(\tilde{w})$ the number of $\overline 1$'s.  

Note that two adjacent vertices in $G_{\pi}$ which are both colored by $\overline 1$'s create an inversion pair since the $\overline 1$ in the higher row standardizes to a larger integer than the one in the lower row.   Whereas, two $1$'s which color adjacent vertices in $G_{\pi}$ never create an inversion pair.  Furthermore, a pair of adjacent vertices $b,c$, with $b$ colored $1$ and $c$ colored $\overline 1$, create an inversion pair if and only if $b>c$.  A moments thought shows that any row containing a $\overline 1$ contributes inversion pairs for each edge of $G_{\pi}$ in its row, while rows with a $1$ do not create any inversion pairs. 
So, given a coloring of vertices $2,3,\ldots ,n$ by $1$'s and $\overline 1$'s,
the cases of coloring vertex $1$ with $1$ and  $\overline 1$ would together contribute 
the weight  $(q^{\alpha} -q^{a_i(\pi)})$ to the right hand side of \eqref{vector}.
Hence by induction we have
\begin{equation}\label{secondstep}
(q-1)^n \chi _{\pi}^{(\alpha)}[1;q]  = \prod_{i=1}^n (q^{\alpha} -q^{a_i (\pi)}),
\end{equation}
and note that 
\begin{align*}
\prod_{i=1}^n (q^{\alpha} -q^{a_i (\pi)}) = q^{\area(\pi)} \prod_{i=1}^n (q^{\alpha-a_i} -1).
\end{align*}
Using this in (\ref{secondstep}) and dividing both sides by $(q-1)^n$ completes the proof of the first equality in \eqref{QQ}.

To prove the second inequality, use the well-known fact that $\chi _{\pi}[X;q]$ is invariant under reversing $\pi$, which replaces row lengths $a_i(\pi)$ by column lengths $b_i(\pi)$. We omit the details.
\end{proof}

\begin{remark}\label{rem:2}
Recently, Hikita \cite{Hikita} announced a proof of the Stanley-Stembridge conjecture.  His approach uses a probabilistic argument motivated from the observation that 
\begin{equation}\label{eqn:sumisone}
\sum_{\lambda\vdash n}q^{|\mathsf{e}|-|\mathsf{e}_{\lambda}|}\frac{c_{\lambda, \pi}(q)}{\prod_i [\lambda_i]_q !}=1,
\end{equation}
where $\chi_{\pi}[X;q]=\sum_{\lambda\vdash n}c_{\lambda,\pi}(q) e_{\lambda}[X]$ is the $e$-expansion of the chromatic quasisymmetric function indexed by a Dyck path $\pi$,
$|\mathsf{e}|={n \choose 2} -\text{area}(\pi)$ is the number of cells above $\pi$, and $|\mathsf{e}_\lambda|=\sum_{i<j}\lambda_i \lambda_j$.
We can easily derive \eqref{eqn:sumisone} from \eqref{QQ}, by applying the principal specialization of the elementary symmetric functions and  by letting $\alpha\rightarrow \infty$.
\end{remark}

\subsection{Rook theory}
For a nonnegative integer sequence $c_1,\dots,c_n$, we associate a \emph{board} $B(c_1,\dots,c_n)$ which is the collection of all cells $(i,j)$ such that $1\le i \le c_i$ and $1\le j \le n$. If the sequence satisfies $0\le c_1\le c_2 \le \cdots \le c_n\le n$, then $B=B(c_1,\dots,c_n)$ is called a \emph{Ferrers board}. 
We say that we \emph{place $k$ nonattacking rooks on $B$} for choosing a $k$-subset of cells in $B$ such that 
no two elements have a common coordinate, that is, no two rooks lie in the same row or in the same column. 
We identify a permutation $\sigma\in\mathfrak{S}_n$ with the  rook placement 
$\{(i,\sigma_i)\,|\,i=1,2,\dots,n\}$ on $[n]\times [n]$. 

For a Ferrers board $B=B(c_1,\dots,c_n)$, let $H_k(B)$ denote the set of $n$-rook placements on $[n]\times [n]$ with $k$ rooks in $B$, i.e.,
\[
    H_k(B)= \{\sigma \in \mathfrak{S}_n: |\sigma \cap B| = k\}
\]
and $h_k(B)$ be the the $k$-th \emph{hit number} defined by $h_k(B)=|H_k(B)|$. Let $R_k(B)$ denote the set of 
all $k$-rook placements in $B$
and $r_k(B)$ be the $k$-th \emph{rook number} defined by $r_k(B)=|R_k(B)|$. The famous Goldman--Joichi--White formula is a factorization formula for the polynomials associated to the rook numbers.
\begin{thm}\label{thm: Garsia-Remmel}\cite{GJW75} Let $B=B(c_1,\dots,c_n)$ be a Ferrers board in $[n]\times [n]$. Then 
\begin{equation}\label{eq: GJW}
    \prod_{i=1}^n (\alpha+c_i-(i-1))=\sum_{k=0}^n r_{n-k}(B) (\alpha)^{\underline{k}}
   \end{equation}
\end{thm}
The right-hand-side of (\ref{eq: GJW}) can also be expressed as $\sum_{k=0}^n \binom{\alpha +k}{n}h_k(B)$.
We refer the product on the left-hand side of \eqref{eq: GJW} as a \emph{rook product}. Generally, a product of the form
\( \prod_{i=1}^n (\alpha + b_i)\)
is called a rook product when the nonincreasing rearrangement $\beta_1,\dots,\beta_n$ of $b_1,\dots,b_n$ satisfies the condition $\beta_i - 1 \le \beta_{i+1}$ for $1\le i \le n-1$,
and $0\le \beta_1\le n$.

Garsia--Remmel introduced  a $q$-analogue of the Goldman--Joichi--White formula \eqref{eq: GJW}. To elaborate, let us define $q$-statistics first. Given a Ferrers board $B$ and
a $k$-rook placement $C$, a placement of $k$ nonattacking rooks on $B$ (see Figure \ref{Ferrers} for an example), let $\inv_B(C)$ denote the number of cells of $B$  which are left after cancelling
all cells which are either weakly to the right in the row, or below in the column, of any rooks in $C$. Then we define 
\begin{align*}
r_k(B;q) =\sum_{C \in R_k(B)} q^{\inv_B(C)}.
\end{align*}
\begin{thm}\label{thm:GR}\cite{GR86} Let $B=B(c_1,\dots,c_n)$ be a Ferrers board in $[n]\times [n]$. Then 
\begin{equation*}
\label{eq: qGJW}
   \prod_{i=1}^n [\alpha+c_i-(i-1)]_q=\sum_{k=0}^n r_{n-k}(B;q) [\alpha]^{\underline{k}}_q
\end{equation*}
\end{thm}

Garsia and Remmel also defined $q$-hit numbers $h_k(B;q)$ algebraically in terms of the $r_k(B;q)$
\[
    \sum_{k=0}^n [k]_q! r_{n-k}(B;q)\prod_{i=k+1}^n(x-q^i) = \sum_{j=0}^n h_j(B;q)x^j. 
\]
They also posed the problem of finding a combinatorial statistic $\stat$ on rook placements $C$ satisfying
\begin{align*}
\label{qhitstat}
h_k(B;q) = \sum_{ C \in H_k(B)}  q^{\stat_B(C)}.
\end{align*}
Dworkin and the first author independently solved this problem  \cite{Dwo98}, \cite{Hag98a}.  Their solutions are closely related to each other.  Here is one way to describe the solution of the first author.  Given a placement of $n$ rooks on $[n]\times [n]$, first cancel each cell in the row to the right of any rook.  Then for each rook on
$B$, count the number of uncancelled cells above it in the column, but still on the board $B$.  Next, for each rook not on $B$ count the number of uncancelled cells as you travel upwards from the rook to the top of  $[n]\times [n]$, then wrap around to the bottom of the column (as on a torus), and continue travelling upwards to the top of the column in $B$.  The statistic $\stat(C)$ is the total number of uncancelled cells counted in this procedure.
For example, for the placement of $n$ rooks on $[n]\times [n]$ on the right in Figure \ref{Ferrers}, $\stat_B(C)=10$ (represented by circles).  As noted by Dworkin and the first author, 
for any Ferrers board $B(c_1,\ldots ,c_n)$ we have
\[
 \prod_{i=1}^n [\alpha+c_i-(i-1)]_q=\sum_{k=0}^n h_k(B;q) {\footnotesize \begin{bmatrix}\alpha+k\\n\end{bmatrix}_q}.
\]

\begin{figure}[!ht]
\begin{tikzpicture}[scale=.54]
 \foreach \i in {0,...,6} {
 \draw[color=gray] (0,\i)--(6,\i);
  \draw[color=gray] (\i,0)--(\i,6);
  }
   \filldraw (1.5,2.5) circle (5pt);
     \filldraw (3.5,3.5) circle (5pt);
       \filldraw (2.5,1.5) circle (5pt);
         \filldraw (4.5,.5) circle (5pt);
 \draw[very thick] (0,0)--(0,2)--(1,2)--(1,3)--(2,3)--(2,5)--(5,5)--(5,6)--(6,6)--(6,0)--(0,0)--cycle;
 \node at (1.5, .5) {\cross};
  \node at (1.5, 1.5) {\cross};
    \node at (2.5, .5) {\cross};
    \node at (2.5, 2.5) {\cross};
      \node at (3.5, .5) {\cross};
        \node at (3.5, 1.5) {\cross};
    \node at (3.5, 2.5) {\cross};
      \node at (4.5, 1.5) {\cross};
        \node at (4.5, 2.5) {\cross};
    \node at (4.5, 3.5) {\cross};
           \node at (5.5, .5) {\cross};
       \node at (5.5, 1.5) {\cross};
        \node at (5.5, 2.5) {\cross};
    \node at (5.5, 3.5) {\cross};
\end{tikzpicture}
\qquad\qquad\quad
\begin{tikzpicture}[scale=.54]
 \foreach \i in {0,...,6} {
 \draw[color=gray] (0,\i)--(6,\i);
  \draw[color=gray] (\i,0)--(\i,6);
  }
  \foreach \i in {2,...,5}{
      \draw[thick] (.5, \i.5) circle (5.5pt);
      }
   \filldraw (.5,1.5) circle (5pt);
     \filldraw (1.5,3.5) circle (5pt);
       \filldraw (2.5,.5) circle (5pt);
         \filldraw (3.5,4.5) circle (5pt);
                  \filldraw (4.5,2.5) circle (5pt);
                           \filldraw (5.5,5.5) circle (5pt);
    \draw[very thick] (1,0)--(1,1)--(2,1)--(2,3)--(3,3)--(3,4)--(6,4)--(6,0)--(1,0)--cycle;
      \draw[thick] (1.5, .5) circle (5.5pt);
          \draw[thick] (1.5, 4.5) circle (5.5pt);
              \draw[thick] (1.5, 5.5) circle (5.5pt);
                  \draw[thick] (2.5, 2.5) circle (5.5pt);
                           \draw[thick] (3.5, 2.5) circle (5.5pt);
                                    \draw[thick] (3.5, 5.5) circle (5.5pt);
   \foreach \i in {0,...,4}{
        \node at (5.5, \i.5) {\cross};
        }
          \node at (4.5, .5) {\cross};
              \node at (4.5, 1.5) {\cross};
              \node at (4.5, 3.5) {\cross};
                    \node at (4.5, 4.5) {\cross};
                          \node at (3.5, .5) {\cross};
                                  \node at (3.5, 1.5) {\cross};
                                          \node at (3.5, 3.5) {\cross};
                                                  \node at (1.5, 1.5) {\cross};
                                                        \node at (2.5, 1.5) {\cross};
                                            \node at (2.5, 3.5) {\cross};
         \end{tikzpicture}
\caption{On the left, a placement of $4$ nonattacking rooks (denoted by bullets) on $B(2,3,5,5,5,6)$ with $\inv=8$.  On the right, a placement of $6$ rooks on $[n]\times [n]$ with
$2$ rooks on $B(0,1,3,4,4,4)$, where $\stat_{B}=10$.}
\label{Ferrers}
\end{figure}
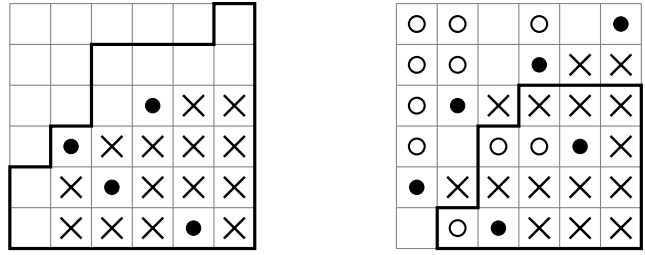

\section{Monomial expansion into $\footnotesize{\begin{bmatrix}\alpha+k\\n\end{bmatrix}_q}$ basis}\label{Sec: binomial alpha+k choose n basis expansion}
\subsection{Monomial expansion}\label{subsec: monomial expansion}
Given an $n$-Dyck path \( \pi \), we define an ordering \( \prec_\pi \) for \emph{bi-letters} $(i,j)\in \mathbb{Z}_{> 0}\times [n]$ as follows:
\[
  (i,j) \prec_\pi (i',j') \iff
  \begin{cases}
    &  i < i' ~\text{ or }~\\
    & i = i'~ \text{ and } ~ j<_\pi j'.
    \end{cases}
\]
For $\lambda \vdash n$, let $M(\lambda)$ be the set of all words 
of {\it weight} $\lambda$, which means all words where $i$ occurs $\lambda_i$ times for $i\ge 1$. For a bi-word \(\displaystyle (w, \sigma) \in M(\lambda) \times \mathfrak{S}_n\), we say that \( i\in [n-1] \)
is an \emph{ascent} of \( (w, \sigma) \) with respect to \( \pi \) if
\[ \displaystyle (w_i, \sigma_i) \prec_\pi (w_{i+1}, \sigma_{i+1}) .\]
We define \( \asc_\pi (w, \sigma) \) as the number of ascents of \( (w, \sigma) \) with respect to \( \pi \). 
Let $\comaj_\pi (w,\sigma)$ be the sum of positions where the ascents of $(w,\sigma)$ occur. Finally, for a bi-word $(w,\sigma)\in M(\lambda)\times\mathfrak{S}_n$, we define a statistic $\stat_\pi$ by
\[
    \stat_\pi(w,\sigma)={\inv_\pi(\tilde{w}) +\sum_{i=1}^{\ell(\lambda)}\inv_{\pi[w^{-1}(\{i\})]}(\sigma[w^{-1}(\{i\})]) + \binom{n}{2}-nk+\comaj_\pi(w,\sigma)},
\]
where \( \tilde{w}(i)=w_{(\sigma^{-1})_i} \), $w^{-1}(\{i\}):=\{j\in[n]:w_j=i\}$, $\pi[A]$ is the induced subgraph $\pi$ on $A$, and $\sigma[A]$ is the restriction of permutation $\sigma$ to $A$. Using these notions, we state a monomial expansion of $\chi^{(\alpha)}_\pi[X;q]$ into the basis $\left\{\footnotesize{\begin{bmatrix}\alpha+k\\n\end{bmatrix}_q}\right\}_{0\le k \le n-1}$. 
 
\begin{thm}\label{thm: alpha+k choose n basis monomial expansion} For an $n$-Dyck path $\pi$, we have
\[
  \chi_\pi^{(\alpha)}[X;q] = \sum_{\lambda\vdash n} \sum_{(w, \sigma)\in M(\lambda) \times \mathfrak{S}_n } q^{\stat_\pi(w,\sigma)}
  \footnotesize{\begin{bmatrix}\alpha+\asc_\pi(w,\sigma)\\n\end{bmatrix}_q} m_\lambda [X].
\]  
\end{thm}

Before we provide a proof of Theorem~\ref{thm: alpha+k choose n basis monomial expansion} in Section~\ref{subsec: proof of Theorem A}, we demonstrate an example of the monomial expansion given in Theorem~\ref{thm: alpha+k choose n basis monomial expansion} first.

\begin{exam}
    Let $\pi=(2,3,3)$ be a Dyck path. For brevity, we let $q=1$ and consider the $\alpha$-chromatic symmetric function $\chi^{(\alpha)}_\pi[X;1]$. There are 18 bi-words $(w,\sigma)\in M(2,1)\times \mathfrak{S}_3$, and the ascents of those bi-words are given as follows.
\begin{table}[h]
\centering
\begin{tabular}{|c|c|c|c|c|c|c|c|c|c|c|c|c|c|c|c|c|c|c|}
\hline
$w$ & 112 & 112 & 112 & 112 & 112 & 112 & 121 & 121 & 121 & 121 & 121 & 121 & 211 & 211 & 211 & 211 & 211 & 211 \\ \hline
$\sigma$ & 123 & 213 & 231 & 312 & 321 & 132 & 123 & 132 & 213 & 231 & 312 & 321 & 123 & 132 & 213 & 231 & 312 & 321 \\ \hline
$\text{asc}_\pi$ & 1 & 1 & 1 & 1 & 1 & 2 & 1 & 1 & 1 & 1 & 1 & 1 & 0 & 0 & 1 & 0 & 0 & 0 \\ \hline
\end{tabular}
\end{table}

Therefore, the \( m_{(2,1)} \) coefficient of \( \chi^{(\alpha)}_\pi[X;1] \) is
\[
\binom{\alpha + 2}{n} + 12 \binom{\alpha + 1}{n} + 5 \binom{\alpha}{n}.
\]
Further computation yields the following expression for \( \chi^{(\alpha)}_{\pi}[X;1] \).
\begin{align*}
    \chi^{(\alpha)}_{\pi}[X;1] =\left(  6\binom{\alpha+2}{n}+24\binom{\alpha+1}{n}+6\binom{\alpha}{n} \right)m_{(1,1,1)}+ \left(  \binom{\alpha +2}{n} +12\binom{\alpha+1}{n}+5\binom{\alpha}{n} \right)m_{(2,1)} +\left(  2\binom{\alpha +1}{n}+4\binom{\alpha}{n} \right)m_{(3)}. 
    \end{align*}
\end{exam}

\subsection{Proof of Theorem~\ref{thm: alpha+k choose n basis monomial expansion}}\label{subsec: proof of Theorem A}
Recall the definition of the unicellular LLT polynomials 
\[
\LLT_{\pi}[X;q]=\sum_{c\in \mathbb{Z}_{>0}^n}q^{\inv_{\pi}(c)}x^c.
\]
We can derive an expression for $\LLT_{\pi}[(t-1)X;q]$ by an extension of our earlier derivation of (\ref{QQ}).  Consider the ordering 
\[
1<\overline{1}<2<\overline{2}<\cdots < n <\overline{n}
\]
on $\mathcal{A}_{\pm}$.
By letting $f=\LLT_\pi[X]$, $X=tX$ and $Y=\epsilon X$ in \eqref{eqn:superization} and by using  \eqref{TStandard},
we have
\begin{equation}\label{Eq: LLT superization1}
    \LLT_\pi[(t-1)X;q] = \sum_{\tilde c\in \mathcal A _{\pm}^n}
    q^{\inv_\pi(\tilde c)} t^{\text{pos}(\tilde c)}(-1)^{\text{neg}(\tilde c)}
    x^{|\tilde c |},
\end{equation}
where $|\tilde c|$ is the positive word obtained by replacing each $\overline k$ in $\tilde c$ by its absolute value $k$. 
To compute the right hand side of \eqref{Eq: LLT superization1}, consider the 
superfillings $\tilde{c}$ with fixed absolute value $c=(c_1,\dots, c_n)$.
Suppose that we fixed the values of $\tilde{c}_k$, for $2\le k\le n$, and 
let us consider the cases when $\tilde{c}_1$ is $c_1$ or $\overline{c}_1$.
If we take $\tilde{c}_1$ to  be $c_1$, then it does not contribute any inversions since it does not form inversion pairs with any occurrences of $c_1$ or $\overline{c}_1$ in rows below. However, if we take $c_1$ to be $\overline{c}_1$,
then any attacking pairs with label $c_1$ or $\overline{c}_1$ in below rows would contribute inversions. Inversions involving other numbers which are neither $c_1$ nor $\overline c_1$ are not affected by this sign changes.  As a result, we get 
the factor $(t-q^{N_\pi(1)})$ where the term $t$ corresponds to the choice of $c_1$ and $-q^{N_\pi(1)}$ to the choice of $\overline c_1$, where  $N_\pi (i)$ is the number of $j$'s such that $i<j$, $i\nless_\pi j$ and $c_i=c_j$, for $1\le i <n$.   By 
induction we get
\begin{equation*}\label{Eq: LLT superization}
    \LLT_\pi[(t-1)X;q] = \sum_{c\in\mathbb{Z}_{>0}^n}q^{\inv_\pi(c)}\prod_{i=1}^n (t-q^{N_\pi(i)})x^c.
\end{equation*}
By replacing $t$ by $q^\alpha$ and dividing
by \( (q-1)^n \), what we get in the left hand side is the $\alpha$-chromatic symmetric function by Proposition~\ref{prop: plethystic relation between X and LLT Carlsson Mellit}. Therefore, we have
\begin{equation}\label{eq: LLT m}
\chi_{\pi}^{(\alpha)} [X;q]=\frac{\LLT_{\pi}[(q^\alpha-1)X;q]}{(q-1)^n}
=\sum_{\lambda\vdash n}\sum_{c\in M(\lambda)}
 q^{{\inv_\pi}(c)}\prod_{i=1}^{\ell(\lambda)}\left( q^{{\area}(\pi[c^{-1}(i)])}
 \prod_{j=1}^{\lambda_i}[\alpha-a_j(\pi[c^{-1}(i)])]_q\right)m_\lambda [X].
\end{equation}

Now we simplify the above identity \eqref{eq: LLT m}. For a given positive integer $\alpha$, a word (coloring) $c\in \mathbb{Z}_{>0}^n$ and  another word (decoration) $d\in {\{0,1,\dots, \alpha-1\}}^n$, we say that the bi-word $(c,d)$ is an \emph{$\alpha$-decorated proper coloring} of $\pi$ of weight $\lambda$ if the weight of $c$ is $\lambda$, and for any cell $(i,j)$ with $i<j$ and $i\nless_\pi j$, we have $(c_i,d_i)\neq (c_j,d_j)$. Let $\Colo^{(\alpha)}_{\pi,\lambda}$ denote the set of $\alpha$-decorated proper colorings of weight $\lambda$. By Proposition~\ref{prop: q-chromatic} and
the definition of the chromatic symmetric function in \eqref{def:chromsym}, we have
\begin{equation*}
  \prod_{j=1}^n [\alpha-a_j(\pi)]_q =q^{-\area(\pi)} \sum_{\substack{\sigma\in \{0,1,\dots,\alpha-1\}^n \\ \text{ proper coloring of $\pi$}}} q^{\inv_\pi(\sigma)+|\sigma|}.
\end{equation*}
Utilizing \eqref{eq: LLT m} and the identity above, we get the monomial expansion of the $\alpha$-chromatic symmetric function as follows:
\begin{equation}\label{eqn:monoex_1}
\chi^{(\alpha)}_\pi [X;q] = \sum_{\lambda\vdash n }\sum_{(c,d)\in \Colo^{(\alpha)}_{\pi,\lambda}}  q^{\inv_\pi (c)+|d|}\prod_{i=1}^{\ell(\lambda)}
q^{\inv_{\pi[c^{-1}(i)]}(d(c^{-1}(i)))} m_\lambda[X].
\end{equation}

To prove Theorem~\ref{thm: alpha+k choose n basis monomial expansion}, it suffices to verify the identity for $n$ distinct values of $\alpha$. To that end, we fix an integer $\alpha$ and prove it by using the bijection below.

\begin{defn}\label{def: the map Phi}

Consider an integer $0\le \alpha \le n-1$. Let $\binom{[\alpha+k]}{n}$ denote the multiset of size $n$ whose elements are in $\{ 0, 1, \dots, \alpha+k-n\}$, namely  
\[
\binom{[\alpha+k]}{n} = \{ \tau=(\tau_1, \tau_2,\dots, \tau_n) ~ | ~ 0\le \tau_i\le \alpha+k-n,\,  \tau_1\le \cdots \le \tau_n\}.
\]
Note that the size of this set is equal to $\binom{\alpha+k}{n}$. 
The map
\[ 
    \Phi:\bigcup_{k=0}^{n-1} \left\{(w,\sigma)\in M(\lambda)\times \mathfrak{S}_n:\asc(w,\sigma)=k\right\} \times \binom{[\alpha+k]}{n} \rightarrow \Colo^{(\alpha)}_{\pi,\lambda}
\]
is defined by
\[
    \Phi(w,\sigma,\tau)_{\sigma_i} = (w_i, \tau_i+(i-1)-\asc_\pi^{<i}(w,\sigma)).
\]
Here, $\asc_\pi^{<i}(w,\sigma)$ counts the ascents occurring in the first $i-1$ positions.
\end{defn}

\begin{prop}
    The map $\Phi$ defined in Definition~\ref{def: the map Phi} is a bijection.
\end{prop}

\begin{proof}
Given an $\alpha$-decorated proper coloring $(c, d) \in \Colo^{(\alpha)}_{\pi, \lambda}$, let $\sigma_i$ be the index $j$ such that $(c_i, d_i)$ is the $j$-th element in the colexicographic order among $\{(c_j,d_j)\}_{j\in[n]}$. This defines a unique permutation $\sigma\in\mathfrak{S}_n$. Let $w = (w_1, \dots, w_n)$ be the word where $w_i = c_{\sigma_i}$. Finally, let $\tilde{d}$ be the weakly increasing rearrangement of $d$, that is, \( \tilde{d}_i=d_{\sigma_i} \). Then $\tau = (\tau_1, \dots, \tau_n)$ is given by
\[
\tau_i = \tilde{d}_i - (i - 1) + \asc_\pi^{<i}(w, \sigma).
\]
It is straightforward to verify that $\Phi(w, \sigma, \tau) = (c, d)$, thus the map $\Phi$ is onto. Through this process, the permutation $\sigma$, the word $w$, and $\tau$ are uniquely defined, demonstrating that the map $\Phi$ is one-to-one.

\end{proof}

\begin{proof}[Proof of Theorem~\ref{thm: alpha+k choose n basis monomial expansion}]

Via the bijection $\Phi$, we can identify $\Colo^{(\alpha)}_{\pi, \lambda}$ with the corresponding set of  bi-words, and so the monomial coefficient in \eqref{eqn:monoex_1} 
can be rewritten as
\begin{equation*}\label{eqn:monoex_2}
[m_\lambda]\left(\chi^{(\alpha)}_\pi [X;q]\right) = 
\sum_{k=0}^{n-1}\sum_{\substack{(w, \sigma,\tau)\in (M(\lambda) \times \mathfrak{S}_n)\times \binom{[\alpha+k]}{n}\\ \asc_\pi(w,\sigma)=k}}q^{\inv_\pi(\tilde{w})+|\tilde{\tau}|}  \prod_{i=1}^{\ell(\lambda)} q^{\inv_{\pi[w^{-1}(\{i\})]}(\tilde{\tau}[w^{-1}(\{i\})])}, 
\end{equation*}
where \( \tilde{w}_i=w_{(\sigma^{-1})_i} \), $\tilde{\tau}=(\tilde{\tau}_1,\dots, \tilde{\tau}_n)$ and $\tilde{\tau}_i =\tau_i +i-1- \asc_\pi^{<i}(w,\sigma)$.

First, we claim that given $(w, \sigma)\in M(\lambda)\times \mathfrak{S}_n$ with $\asc_\pi (w,\sigma)=k$, the factor 
\[
q^{\inv_\pi (\tilde{w})}  \prod_{i=1}^{\ell(\lambda)} q^{\inv_{\pi[w^{-1}(\{i\})]}(\tilde{\tau}[w^{-1}(\{i\})])}
\]
is invariant under varying $\tau \in \binom{[\alpha +k]}{n}$. The term $q^{\inv_\pi (\tilde{w})}  $ is clearly independent of $\tau$, so we only need to care about the product $\prod_{i=1}^{\ell(\lambda)} q^{\inv_{\pi[w^{-1}(\{i\})]}(\tilde{\tau}[w^{-1}(\{i\})])}$. First, observe that 
$$\inv_{\pi[w^{-1}(\{i\})]}(\tilde{\tau}[w^{-1}(\{i\})]) = \inv_{\pi[w^{-1}(\{i\})]}(\sigma[w^{-1}(\{i\})]),$$ for each $i$, $1\le i \le \ell(\lambda)$. This is due to the way we define \(\sigma  \) according to the colexicographic order on \( (c,d) \in \mathcal{C}_{\pi, \lambda}^{(\alpha)} \), the relative order of the values of \( \sigma \) and \( \tilde{\tau} \) are the same within the set of \( w^{-1}(\{i\}) \)'s.
In $\sigma[w^{-1}(\{i\})]$, the inversion of \( \pi[w^{-1}(\{i\})] \) occurs when $\tilde{\tau}_l <\tilde{\tau}_m$ and $\sigma_l >\sigma_m$ 
for $l<m$ and $l\nless_{\pi} m$. Since we are considering $w^{-1}(\{i\})$, if $l\nless_{\pi} m$, 
the inequality $\tilde{\tau}_l <\tilde{\tau}_m$ holds for sure, so we only need to count the number of pairs such that 
$\sigma_l >\sigma_m$ for $l<m$ and $l\nless_{\pi} m$. Hence, we can factor out the term \(\displaystyle q^{\inv_\pi(\tilde{w})}  \prod_{i=1}^{\ell(\lambda)} q^{\inv_{\pi[w^{-1}(\{i\})]}(\sigma[w^{-1}(\{i\})])}\) from the summation over $\tau \in \binom{[\alpha+k]}{n}$.

Now we take care of the rest of the factor 
\[
\sum_{\tau\in\binom{[\alpha+k]}{n}}q^{|\tilde{\tau|}}
= \sum_{\tau\in\binom{[\alpha+k]}{n}}q^{\sum_{i=1}^n \left(\tau_i +i-1-\asc^{<i}_\pi(w,\sigma) \right)}
= q^{\sum_{i=1}^n\left( i-1-\asc_\pi^{<i}(w,\sigma)\right) } \sum_{\tau\in\binom{[\alpha+k]}{n}}q^{|\tau|},
\]
given $(w, \sigma)\in M(\lambda)\times \mathfrak{S}_n$ with $\asc_\pi (w,\sigma)=k$. 
For a given sequence $\tau =(\tau_1,\dots, \tau_n)$, with $0\le \tau_i \le \alpha+k-n$, $\tau_1 \le \tau_2\le \cdots \le \tau_n$, we can identify it with a lattice path contained in a rectangular shape of size $n\times(\alpha+k-n)$ 
with the $i$-th column height $\tau_i$. Then it is not very hard to see that 
\[
 \sum_{\tau\in\binom{[\alpha+k]}{n}}q^{|\tau|}= \footnotesize{\begin{bmatrix}\alpha+k\\n\end{bmatrix}_q}.
\]
Lastly, via a simple computation, we get 
\(\displaystyle 
q^{\sum_{i=1}^n\left( i-1-\asc^{<i}_\pi (w,\sigma)\right) } 
=q^{\binom{n}{2}-nk+\comaj_{\pi}(w, \sigma)},\) for a given bi-word  $(w, \sigma)\in W(\lambda)\times \mathfrak{S}_n$ with $\asc_\pi (w,\sigma)=k$. 
Thus we obtain Theorem~\ref{thm: alpha+k choose n basis monomial expansion}.
\end{proof}

\section{Expansion into $(\alpha)^{\underline{k}}$ basis}\label{Sec: bracket alpha choose k basis}

\subsection{The XY technique}\label{Sec:XYTechnique}
Given two sets of variables  $X$ and $Y$, let $XY = \{x_iy_j:i,j\ge 1\}$ and $\pi$ be an $n$-Dyck path.  
\begin{thm}
\label{XYThm}
\begin{align}
\label{XYresult}
\chi _{\pi}[XY;q] = \sum_{\lambda \vdash n } m_{\lambda}[X]
\sum_{w \in M(\lambda)} 
q^{\inv_{\pi}(w)} \chi _{\beta (\pi,w)} [Y;q],
\end{align}
where the $n$-Dyck path $\beta(\pi,w)$ corresponds to the graph $G_{\beta}$ obtained by starting with $G_{\pi}$ colored by $w$, and then
removing all edges which connect  two vertices $a,b$ with different colors 
(i.e. vertices satisfying $w_a\ne w_b$).
\end{thm}
\begin{proof} We use the lexicographic order on ordered pairs of positive integers $(a,b)$, that is:
\begin{align}
\label{order}
\text{if $a<c$ then $(a,b)<(c,d)$ for all $b,d$ and also
$(b,a)<(b,c)$ for all $b$}.
\end{align}
By definition,
\begin{align}
\label{DefXY}
\chi _{\pi} [XY;q] = 
\sum _{C \in (\mathbb Z_{>0}^2)^n \atop C \text{ proper } }
q^{\inv_{\pi}(C)} 
\prod_{i\ge 1} x_{w_{i}}y_{z_{i}},
\end{align}
where $C_i=(w_i,z_i)$ with both $w$ and $z$ in $\mathbb Z_{>0}^n$, and we use (\ref{order}) to determine inequalities among the $C_i$.  Now for any  $i <j$ where $(i,j) \in G_{\pi}$,  if $w_{i}\ne w_{j}$, then
whether or not $C_i$ and $C_j$ contribute to $\inv_{\pi}(C)$ is completely determined by the values of $w_{i}$ and $w_{j}$.  On the other hand, if 
$w_{i}=w_{j}$, then
whether or not $C_i$ and $C_j$ contribute to $\inv_{\pi}(C)$ is completely determined by the values of $z_{i}$ and $z_{j}$, and in this case we also must have 
 $z_{i}\ne z_{j}$, otherwise the coloring $C$ of $G_{\pi}$ would not be proper.  It follows that 
 \begin{align*}
 \inv_{\pi}(C) = \inv_{\pi}(w_{1},w_{2},\ldots ,w_{n}) + 
 \inv_{\beta (\pi,w)}(z_{1},z_{2},\ldots ,z_{n}),
 \end{align*}
 where the first term on the right counts the inversions amongst $C_i$ and $C_j$ which have unequal first coordinates, and 
  $ \inv_{\beta (\pi,w)}$ counts the inversions amongst $C_i$ and $C_j$ which have equal first coordinates.  Thus the contribution to $\chi _{\pi}[XY;q]$ from all
  colorings $C$ whose first coordinate $w$ has weight $\lambda$ on the right-hand-side of (\ref{DefXY}) factors into a term of weight $\lambda$ in the $X$ variables, times a sum of terms in the $Y$ variables, which are over proper colorings of $G_{\beta (\pi,w)}$.    We leave it to the reader to show that $G_{\beta(\pi,w)}$ is a unit interval order when 
  $G_{\pi}$ is.  See Example \ref{XYex} for an illustration of this construction for a given $\pi$ and coloring $C$.
\end{proof}

\begin{exam}
\label{XYex}
Let $\pi$ have  Hessenberg function  $(3,3,4,6,6,6,8,9,9)$, as on the left in Figure \ref{XYfig} below.  Here we have labelled the rows and columns to the left and above the figure.  The squares of the graph $G_{\pi}$ here are $\{(2,1),(3,2),(5,4),(6,4),(6,5),(7,6),(8,7),(9,7),(9,8)\}$.
We have placed the ordered pair designating the
color of  vertex $i$, i.e $(w_i,z_i)$, in square $(i,i)$ on the diagonal, so the color of 
vertex $1$ here is $(1,1)$, that of vertex $2$ is $(1,2)$, that of vertex $3$ $(1,1)$, etc..  
The statistic $\inv_{\pi}(C) = 5$, with specific inversions occurring at the edges $(3,2),(6,4),(6,5),(9,7),(9,8)$; 
$\inv_{\pi}(w) = 3$, with edges $(6,4),(6,5),(9,8)$ contributing, and $\inv_{\pi}(z) = 2$, from edges $(3,2),(9,7)$.   
To create $G_{\beta (\pi,w)}$, we remove
edges $(6,4),(6,5),(8,7),(9,8)$, then renumber the vertices so all vertices whose color have the same first coordinate are consecutive. 

\begin{figure}[!ht]
\begin{tikzpicture}[scale=.6]
\draw (0,0)--(0,9)--(9,9);
\foreach \i in {0,...,8}{
\draw (0,\i)--(\i+1,\i);
\draw (\i+1,9)--(\i+1,\i);
}
\foreach \i in {1,...,9}{
\node () at (-.5, 9-\i+.5) {$\i$};
\node () at (9-\i+.5, 9.5) {$\i$};
}
\draw[line width=1.5pt, color=dredcolor] (0,0)--(0,1)
(0,2)--(0,3)--(1,3)
(2,3)--(2,4)--(3,4)
(3,6)--(4,6)
(6,6)--(6,8)--(7,8)--(7,9)--(9,9);
\draw[line width=1.5pt, color=cerulean] (0,1)--(0,2)
(1,3)--(2,3)
(3,4)--(3,6)
(4,6)--(6,6);
\node () at (.5, .5) {\tiny$(1,1)$};
\node () at (1.5, 1.5) {\tiny$(2,3)$};
\node () at (2.5, 2.5) {\tiny$(1,3)$};
\node () at (3.5, 3.5) {\tiny$(1,2)$};
\node () at (4.5, 4.5) {\tiny$(2,4)$};
\node () at (5.5, 5.5) {\tiny$(2,1)$};
\node () at (6.5, 6.5) {\tiny$(1,1)$};
\node () at (7.5, 7.5) {\tiny$(1,2)$};
\node () at (8.5, 8.5) {\tiny$(1,1)$};
\draw[thick, ->] (9,4.5)--(11,4.5);
\end{tikzpicture}
\qquad
\begin{tikzpicture}[scale=.6]
\draw (0,0)--(0,9)--(9,9);
\foreach \i in {0,...,8}{
\draw (0,\i)--(\i+1,\i);
\draw (\i+1,9)--(\i+1,\i);
}
\foreach \i in {1,...,9}{
\node () at (-.5, 9-\i+.5) {$\i$};
\node () at (9-\i+.5, 9.5) {$\i$};
}
\draw[line width=1.5pt, color=dredcolor] (0,0)--(0,2)--(1,2)--(1,3)--(3,3)--(3,5)--(4,5)--(4,6)--(6,6);
\draw[line width=1.5pt, color=cerulean] (6,6)--(6,7)--(7,7)--(7,9)--(9,9);
\node () at (.5, .5) {\tiny$(1,1)$};
\node () at (1.5, 1.5) {\tiny$(1,3)$};
\node () at (2.5, 2.5) {\tiny$(1,2)$};
\node () at (3.5, 3.5) {\tiny$(1,1)$};
\node () at (4.5, 4.5) {\tiny$(1,2)$};
\node () at (5.5, 5.5) {\tiny$(1,1)$};
\node () at (6.5, 6.5) {\tiny$(2,3)$};
\node () at (7.5, 7.5) {\tiny$(2,4)$};
\node () at (8.5, 8.5) {\tiny$(2,1)$};
\end{tikzpicture}
\caption{On the left, a proper coloring of $G_{(3,3,4,6,6,6,8,9,9)}$, and on the right, the corresponding coloring of  $G_{\beta (\pi,w)}$.}
\label{XYfig}
\end{figure}
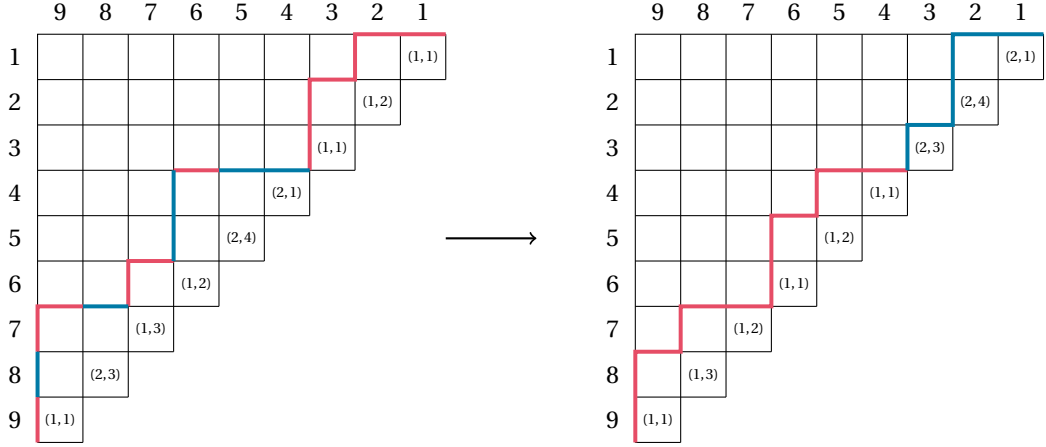

\end{exam}

\begin{cor}\label{Cor: Schur positivity alpha N}
Given an $n$-Dyck path $\pi$, let
\begin{align*}
\chi_{\pi}^{(\alpha)} [X;q] = \sum _{\lambda\vdash n} 
C_{\pi,\lambda}(\alpha) s_\lambda[X].
\end{align*}
Then if $\alpha \in \mathbb N$, $C_{\pi,\lambda}(\alpha) \in \mathbb N [q]$ and furthermore has a combinatorial interpretation.
\end{cor}
\begin{proof}   Let $X=Q_{\alpha}$ in (\ref{XYresult}).  Clearly  $m_{\lambda}(Q_{\alpha})\in \mathbb N[q]$ since $\alpha \in \mathbb N$.
Since all the Schur 
coefficients of the $\chi _{\beta(\pi,w)}(Y;q)$ are in $\mathbb N [q]$ and count weighted $P$-tableaux by the result of Shareshian and Wachs \cite{ShWa16}, 
everything in (\ref{XYresult}) is positive with a combinatorial interpretation.
\end{proof}

\subsection{Set partitions and the $(\alpha)^{\underline{k}}$ expansion}\label{subsec: Theorem B}
A \emph{set partition} of $[n]$ is a collection of nonempty subsets of $[n]$ such that each element in $[n]$ is included in exactly one subset. Each subset in a set partition is called a \emph{part}. The set of set partitions of $[n]$ into $k$ parts is denoted by $\SSS(n,k)$.

For an $n$-Dyck path $\pi$ and a set partition $S$ of $[n]$, we assign a Dyck path $\beta(\pi,S)$ as follows. First, order the parts of $S$ as $S=\{S^{(1)},\dots,S^{(k)}\}$ and let $w_S$ be the word obtained by replacing elements in $S^{(i)}$ with $i$. Now define
\[
    \beta(\pi,S) := \beta(\pi,w_S).
\]
Using this, we present the expansion of $\alpha$-chromatic symmetric functions into $(\alpha)^{\underline{k}}$ basis.

\begin{thm}\label{thm: falling factorial basis} For an $n$-Dyck path $\pi$, the $\alpha$-chromatic symmetric function $\chi^{(\alpha)}_\pi[X;1]$ at $q=1$ is expanded into falling factorial basis $\left\{(\alpha)^{\underline{k}}\right\}_{1\le k \le n}$ as
\begin{equation}\label{Eq: falling factorial expansion}
\chi_{\pi}^{(\alpha)}[X;1] = \sum_{k=1}^n (\alpha)^{\underline{k}}\sum_{S\in \SSS(n,k)} \chi_{\beta(\pi,S)}[X;1].
\end{equation}
In particular, the $\alpha$-chromatic symmetric function \( \chi_{\pi}^{(\alpha)} [X;1] \) is positively expanded in terms of the basis \( \left\{ (\alpha)^{\underline{k}}s_{\lambda} \right\}_{\substack{1\le k \le n, \\ \lambda\vdash n}}\).

\end{thm}
\begin{remark}\label{rem:3}
Due to Hikita's proof of Stanley-Stembridge conjecture \cite{Hikita}, Theorem~\ref{thm: falling factorial basis} gives 
$e$-positivity of the $\alpha$-chromatic symmetric functions when \( q=1 \). 
\end{remark}

Before presenting a proof, 
we give an example demonstrating Theorem~\ref{thm: falling factorial basis}.

\begin{exam}
    Let $\pi$ be a Dyck path $(2,3,3)$. The following drawn with blue lines are Dyck paths $\beta(\pi,S)$ for set partitions $\{\{1,2,3\}\}, \{\{1,2\},\{3\}\}, \{\{1,3\},\{2\}\}, \{\{1\},\{2,3\}\}, \{\{1\},\{2\},\{3\}\}$ of $[3]$, respectively.

    
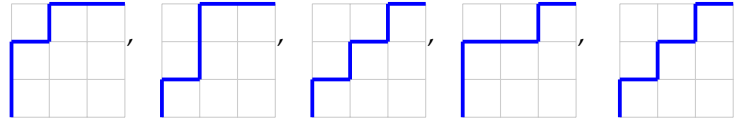
\begin{center}
    \begin{tikzpicture}[scale=0.5,baseline=1cm]
    \draw[black!20] (0,0) grid (3,3);
    \draw[line width=0.5mm,blue] (0,0)--(0,2);    \draw[line width=0.5mm,blue] (0,2)--(1,2);
    \draw[line width=0.5mm,blue] (1,2)--(1,3);
    \draw[line width=0.5mm,blue] (1,3)--(3,3);
\end{tikzpicture},\quad \begin{tikzpicture}[scale=0.5,baseline=1cm]
    \draw[black!20] (0,0) grid (3,3);
    \draw[line width=0.5mm,blue] (0,0)--(0,1);    \draw[line width=0.5mm,blue] (0,1)--(1,1);
    \draw[line width=0.5mm,blue] (1,1)--(1,3);
    \draw[line width=0.5mm,blue] (1,3)--(3,3);
\end{tikzpicture},\quad\begin{tikzpicture}[scale=0.5,baseline=1cm]
    \draw[black!20] (0,0) grid (3,3);
    \draw[line width=0.5mm,blue] (0,0)--(0,1);    \draw[line width=0.5mm,blue] (0,1)--(1,1);
    \draw[line width=0.5mm,blue] (1,1)--(1,2);
    \draw[line width=0.5mm,blue] (1,2)--(2,2);
    \draw[line width=0.5mm,blue] (2,2)--(2,3);
    \draw[line width=0.5mm,blue] (2,3)--(3,3);
\end{tikzpicture},\quad \begin{tikzpicture}[scale=0.5,baseline=1cm]
    \draw[black!20] (0,0) grid (3,3);
    \draw[line width=0.5mm,blue] (0,0)--(0,2);    \draw[line width=0.5mm,blue] (0,2)--(2,2);
    \draw[line width=0.5mm,blue] (2,2)--(2,3);
    \draw[line width=0.5mm,blue] (2,3)--(3,3);
\end{tikzpicture}, \quad \begin{tikzpicture}[scale=0.5,baseline=1cm]
    \draw[black!20] (0,0) grid (3,3);
    \draw[line width=0.5mm,blue] (0,0)--(0,1);    \draw[line width=0.5mm,blue] (0,1)--(1,1);
    \draw[line width=0.5mm,blue] (1,1)--(1,2);
    \draw[line width=0.5mm,blue] (1,2)--(2,2);
    \draw[line width=0.5mm,blue] (2,2)--(2,3);
    \draw[line width=0.5mm,blue] (2,3)--(3,3);
\end{tikzpicture}.
\end{center}

By Theorem~\ref{thm: falling factorial basis}, we have 
\[
    \chi^{(\alpha)}_\pi [X;1]= (\alpha)^{\underline{1}}\chi_{(2,3,3)} + (\alpha)^{\underline{2}}\left(\chi_{(1,3,3)} + \chi_{(1,2,3)} + \chi_{(2,2,3)} \right) + (\alpha)^{\underline{3}}X_{(1,2,3)}.
\]
\end{exam}

\begin{proof}[Proof of Theorem~\ref{thm: falling factorial basis}]
Let $w$ be a word of length $n$. Define $S(w)$ to be the set partition such that $i$ and $j$ belong to the same part of $S(w)$ if and only if $w_i=w_j$. For example, $S(31321)=\{\{1,3\},\{2,5\},\{4\}\}$. 
By letting $X=Q_{\alpha}$ and $Y=X$ in (\ref{XYresult}), we have
\[
    \chi_{\pi}^{(\alpha)}[X;q] = \sum_{\lambda \vdash n } m_{\lambda}[Q_\alpha] 
    \sum_{w \in M(\lambda)} 
    q^{{\inv}_{\pi}(w)} \chi _{\beta (\pi,w)} [X;q],
\]
Next specialize $q=1$ to obtain 
\begin{equation}\label{eq: alpha chromsym at q=1 first equation}
    \chi^{(\alpha)}_\pi[X;1] = \sum_{\lambda\vdash n}\sum_{w\in M(\lambda)} \binom{\alpha}{\ell(\lambda)}\binom{\ell(\lambda)}{m_1(\lambda),m_2(\lambda),\cdots,m_{\ell(\lambda)}(\lambda)} \chi_{\beta(\pi,w)}[X;1],    
\end{equation}
where $m_i(\lambda)$ denotes the multiplicity of $i$ in the partition $\lambda$.

For a set partition $S$ of $n$, one can associate a partition $\lambda(S)$ by reordering the sizes of each part of $S$ to a non-increasing sequence. Let $\SSS(n,\lambda)$ be the set of set partitions $S$ with $\lambda(S)=\lambda$. 

Fix a partition $\lambda\vdash n$ and a set partition $S\in\SSS(n,\lambda)$. Note that the number of words $w$ such that $S(w)=S$ is given by $m_1(\lambda)! m_2(\lambda)!\cdots m_{\ell(\lambda)}(\lambda)!$. Therefore, we can rewrite \eqref{eq: alpha chromsym at q=1 first equation} as 
\begin{align*}
    \chi^{(\alpha)}_\pi[X;1] &= \sum_{\lambda\vdash n}\sum_{S\in \SSS(\lambda)} \binom{\alpha}{\ell(\lambda)}\binom{\ell(\lambda)}{m_1(\lambda),m_2(\lambda),\cdots}m_1(\lambda)!m_2(\lambda)!\cdots m_{\ell(\lambda)}(\lambda)! \chi_{\beta(\pi,S)}[X;1]
    \\&= \sum_{\lambda\vdash n} \sum_{S\in \SSS(\lambda)}(\alpha)^{\underline{\ell(\lambda)}}\chi_{\beta(\pi,S)}[X;1],
\end{align*}
which finishes the proof.
\end{proof}

\begin{remark}
Theorem~\ref{thm: falling factorial basis} can be shown by using the rook theory. We provide a brief overview of the key ideas. Setting $q=1$ in equation \eqref{eq: LLT m}, we obtain the expression
\[ \chi_{\pi}^{(\alpha)} [X;1]=\sum_{\lambda\vdash n}\sum_{\substack{\sigma\in M(\lambda)}} \prod_{i=1}^{\ell(\lambda)} \prod_{j=1}^{\lambda_i}(\alpha-a_j(\pi[w^{-1}(\{i\})])) m_\lambda[X]. \]

By considering lexicographic order on the indices $(i,j)$ in the product $\prod_{i=1}^{\ell(\lambda)} \prod_{j=1}^{\lambda_i}(\alpha-a_j(\pi[w^{-1}(\{i\})]))$, this product can be interpreted as a rook product. Let $B(\pi,\sigma)$ represent the corresponding board. There is a natural correspondence between rook placements and set partitions as follows. Given a $k$-rook placement on a board $B$, let $f:[n]\rightarrow [n]\setminus\{1\}\cup\{*\}$ be a function defined by 
\[
    f(i)=\begin{cases}
        j+1 \text{ if there is no rook at $(i,j)$,}\\
        * \text{ otherwise}.
    \end{cases}
\]
Then the parts of the corresponding set partition are given by orbits of $f$ minus $*$. Through this correspondence between rook placements and set partitions, it becomes evident that the union of sets of $(n-k)$-rooks on the board $B(\pi,\sigma)$, for varying words $\sigma$, is in bijection with the union of proper colorings of $\beta(\pi,S)$ over all set partitions $S\in S(n,k)$. With this in mind, by applying the identity in Thereom \ref{thm: Garsia-Remmel} we have:
\begin{align*}
    \chi_{\pi}^{(\alpha)} [X;1]&
    =\sum_{\lambda\vdash n}\sum_{\substack{\sigma\in M(\lambda)}} \sum_{k=1}^n (\alpha)^{\underline{k}} r_{n-k}(B(\pi,\sigma)) m_\lambda[X]  \\
    &=\sum_{k=1}^n \sum_{S\in S(n,k)} (\alpha)^{\underline{k}} \chi_{\beta(\pi,S)}[X;1]. 
\end{align*} 
Alternatively, Theorem~\ref{thm: falling factorial basis} can be proven by using bijective arguments on $p$-coefficients of $\alpha$-chromatic symmetric function, or by employing the modular law for chromatic symmetric functions. In Appendix~\ref{Sec: appendix}, we outline a proof via the modular law.
\end{remark}

\section{Further results on Schur expansion}\label{Sec: Schur positivity}
\subsection{Schur positivity}

In this section we prove the Schur positivity of the $\alpha$-chromatic symmetric functions in terms of the basis $\left\{\footnotesize{\begin{bmatrix}\alpha+k\\n\end{bmatrix}_q}\right\}_{0\le k \le n-1}$. Notably, we establish a general result for Schur positivity of $s_\lambda[Q_\alpha X]$.

\begin{prop}\label{Lemma: Schur positivity q^a-1/q-1}
For a partition $\lambda$, \(s_\lambda\left[Q_\alpha X\right]\) is Schur positive in terms of the basis $\left\{\footnotesize{\begin{bmatrix}\alpha+k\\n\end{bmatrix}_q}\right\}_{0\le k \le n-1}$.
\end{prop}

As a Corollary, the Schur positivity of the $\alpha$-chromatic symmetric function $\chi^{(\alpha)}_\pi[X;q]$ in terms of $\left\{\footnotesize{\begin{bmatrix}\alpha+k\\n\end{bmatrix}_q}\right\}_{0\le k \le n-1}$ follows from the Schur positivity of the usual chromatic symmetric function \cite{Gas96, ShWa16}. 
\begin{cor}\label{Cor: Schur positivity}
    If a symmetric function $f$ is Schur positive, then \(f\left[Q_\alpha X\right]\) also exhibits Schur positivity in the basis $\left\{\footnotesize{\begin{bmatrix}\alpha+k\\n\end{bmatrix}_q}\right\}_{0\le k \le n-1}$. In particular, the $\alpha$-chromatic symmetric function $\chi^{(\alpha)}_\pi[X;q]$ of an $n$-Dyck path $\pi$ is Schur positive in terms of aforementioned basis.
\end{cor}

\begin{proof}[Proof of Lemma~\ref{Lemma: Schur positivity q^a-1/q-1}]
Let us recall the following identity \cite[Eq (7.9)]{Mac1995}: for a partition $\lambda\vdash n$,
\begin{equation}\label{eq: s[XY] Kronecker}
    s_\lambda[XY]=\sum_{\mu,\nu \vdash n} g^{\lambda}_{\mu,\nu} s_\mu[X]s_\nu[Y],
\end{equation}
where $g^{\lambda}_{\mu,\nu}$ is the \emph{Kronecker coefficient}. Setting $X=Q_\alpha, Y=X$ in \eqref{eq: s[XY] Kronecker}, we obtain
\begin{equation*}\label{eq: s[q^a-1/q-1 X]}
    s_\lambda\left[Q_\alpha X\right] 
    = \sum_{\mu,\nu \vdash n} g^{\lambda}_{\mu,\nu} s_\mu\left[Q_\alpha\right]s_\nu[X].
\end{equation*}
Since the Kronecker coefficients $g^{\lambda}_{\mu,\nu}$ are nonnegative, it is enough to show the positivity of $s_\mu\left[Q_\alpha \right]$ in terms of the basis $\left\{\footnotesize{\begin{bmatrix}\alpha+k\\n\end{bmatrix}_q}\right\}_{0\le k \le n-1}$.

The principal specialization of Schur functions has the following combinatorial formula \cite[Proposition 7.19.12]{StaEC}: for a partition $\lambda\vdash n$,
\begin{equation}\label{eq: Schur principal specialization}
    s_\lambda\left[Q_\alpha \right]=\sum_{T\in\SYT(\lambda)}q^{\maj(T)}
    {\begin{bmatrix}\alpha+n-1-d(T)\\n\end{bmatrix}_q},    
\end{equation}
where $d(T)$ is the number of $i$ such that $i+1$ appears below $i$ in $T$, i.e. the \emph{descents} of the standard tableau $T$. This establishes positivity of $s_\lambda\left[Q_\alpha \right]$ in terms of the basis $\left\{\footnotesize{\begin{bmatrix}\alpha+k\\n\end{bmatrix}_q}\right\}_{0\le k \le n-1}$.

\end{proof}

It is noteworthy that since $\begin{bmatrix}\alpha+k\\n\end{bmatrix}_q$ can be positively expanded in the basis $\left\{\footnotesize{\begin{bmatrix}\alpha\\k\end{bmatrix}_q}\right\}_{1\le k \le n}$ for $0\le k\le n-1$, Corollary \ref{Cor: Schur positivity} also implies the Schur positivity in terms of the basis 
$\left\{\footnotesize{\begin{bmatrix}\alpha\\k\end{bmatrix}_q}\right\}_{1\le k \le n}$. However, 
$s_\lambda\left[Q_\alpha X\right]$ might not be integrally Schur positive in the $q$-falling factorial basis $[\alpha]^{\underline{k}}_q$. For example,  $s_{(2)}\left[Q_\alpha X\right]$ at $q=1$ is
    \[
      \left(\frac{1}{2} \alpha^2 - \frac{1}{2} \alpha\right)s_{(2)} + \left(\frac{1}{2}\alpha^2 + \frac{1}{2}\alpha\right)s_{(1,1)} = \frac{1}{2}(\alpha)^{\underline{2}} s_{(2)}
      +\left( \frac{1}{2}(\alpha)^{\underline{2}}+(\alpha)^{\underline{1}}  \right)s_{(1,1)},
    \]
   in which the Schur coefficients are not integral in terms of the falling factorial basis (even for $q=1$).
Nevertheless, we conjecture that this stronger Schur positivity in terms of the $q$-falling factorial basis holds for the $\alpha$-chromatic symmetric functions.
\begin{conj}\label{conj: falling factorial Schur positivity} 
    For an $n$-Dyck path $\pi$, the $\alpha$-chromatic symmetric function $\chi_\pi^{(\alpha)}[X;q]$ has a positive integral expansion in terms of 
    \(
    \left\{[\alpha]^{\underline{k}}_q s_\lambda \right\}_{1\le k \le n,\lambda\vdash n}.
    \)
    Moreover, the Schur expansion of \( \chi^{(\alpha)}_\pi[X;q] \) is of the  form
    \begin{equation}\label{eqn:LLTconj}
      \chi^{(\alpha)}_\pi[X;q]=\sum_{\lambda\vdash n}\left(  \sum_{T\in \SYT(\lambda')}
      q^{\text{stat}(T)}\text{PR}(B(T)) \right)s_{\lambda}[X]    \end{equation}
    for some \( q \)-power \( \text{stat}(T) \) depending on \( T\in \SYT(\lambda') \), where
     \( \text{PR}(B(T)) \) is the rook product
    \( \text{PR}(B(T))=\prod_{i=1}^n [\alpha +c_i -i +1] \) with the \( i \) th column height given by \( c_i \) for some board \( B(T) \).
  \end{conj}

  \begin{remark}\label{rem:1}
  Recall the \( q \)-analogue of the Goldman-Joichi-White identity
  \[
    \prod_{i=1}^n [\alpha+c_i -i+1]_q =\sum_{k=0}^{n-1} \footnotesize{\begin{bmatrix}\alpha+k\\n\end{bmatrix}_q} h_k (B;q)
    = \sum_{k=1}^{n} [\alpha]_q^{\underline{k}} r_{n-k} (B;q)
  \]
  where $B$ is the board associated to the sequence $c_1,c_2,\dots,c_n$. By this formula,
  the Schur expansion in Conjecture \ref{conj: falling factorial Schur positivity} would also impliy
  the integral positivity in terms of both bases $\left\{\footnotesize{\begin{bmatrix}\alpha+k\\n\end{bmatrix}_q}\right\}_{0\le k \le n-1}$ and \(
    \left\{[\alpha]^{\underline{k}}_q \right\}_{1\le k \le n}\).
    Furthermore, if we let \( \alpha \) approaches infinity in the definition
    \[
    \chi_{\pi}^{(\alpha)} [X;q]=\frac{\LLT_{\pi}[(q^\alpha-1)X;q]}{(q-1)^n},  
    \]
    then we get
    \[
      \LLT_\pi [X;q]=\sum_{\lambda\vdash n}\sum_{T\in\SYT(\lambda)}q^{\text{stat}(T)}s_\lambda[X].
    \]
    Thus, figuring out the \( q \)-statistic in \eqref{eqn:LLTconj} will automatically give
    the Schur expansion of the unicellular LLT polynomials.
  \end{remark}
    As an evidence of Conjecture \ref{conj: falling factorial Schur positivity}, we have the following Schur expansion of \( \chi_\pi[X;q] \)
  when \( \pi=(n,n,\dots, n) \) corresponds to the complete graph. 
 \begin{prop}\label{prop:Kn}
\begin{equation}\label{eq:3}
  \chi_{(n,n,\dots,n)}^{\alpha}[X;q]=\sum_{\lambda\vdash n} \left( \sum_{T\in \SYT(\lambda)}q^{\text{ch}(T)}
  \prod_{u\in \lambda}[\alpha-c(u)]_q \right) s_{\lambda}[X],
\end{equation}
where \( \text{ch}(T) \) is the charge statistic of \( T \).
\end{prop}

\begin{proof}
  Note that \( \LLT_{(n^n)}[X;q]=H_{(n)}(X;q,t) = J_{(n)}\left[ \frac{X}{1-t};q,t \right ]\), where \( H_\lambda (X;q,t) \) is the \emph{modified Macdonald polynomials} and \( J_\lambda (X;q,t) \) is the \emph{integral form Macdonald polynomials} (see \cite{Hag2008} for the details). 
\begin{align*}
 \LLT_{(n^n)}[(t-1)X;q] &= H_{(n)}[X(t-1);q,t]  \\
                        &= J_{(n)}\left[ \frac{X(t-1)}{1-t};q,t \right ]= J_{(n)}[-X;q,t]\\
                        &= \sum_{\lambda\vdash n}\prod_{u\in \lambda}(1-q^{a'(u)-l'(u)}t)\left( \sum_{T\in\SYT(\lambda')}q^{\text{ch}(T)} \right)s_{\lambda}[-X]\\
  &=\sum_{\lambda\vdash n}\prod_{u\in\lambda}(q^{l'(u)-a'(u)}t-1)\left(  \sum_{T\in\SYT(\lambda)}q^{\text{ch}(T)} \right) s_{\lambda}[X].
\end{align*}
So we have,
\begin{align*}
  \chi_{(n^n)}^\alpha (X;q)&=   \sum_{\lambda\vdash n}\prod_{u\in\lambda}\frac{(q^{\alpha-(a'(u)-l'(u))}t-1)}{(q-1)^n}\left(  \sum_{T\in\SYT(\lambda)}q^{\text{ch}(T)} \right) s_{\lambda}[X]
  \\&=\sum_{\lambda\vdash n}\prod_{u\in\lambda}[\alpha-(a'(u)-l'(u))]_q\left( \sum_{T\in\SYT(\lambda)}q^{\text{ch}(T)} \right)s_{\lambda}[X]. 
\end{align*}
\end{proof}

\subsection{Symmetry for $\alpha$-chromatic symmetric functions} 

We present a symmetry relation for the Schur expansion of the $\alpha$-chromatic symmetric functions. 

\begin{prop} For an $n$-Dyck path $\pi$, if we let 
\[
    \chi^{(\alpha)}_\pi[X;q] = \sum_{k=0}^{n-1}\sum_{\lambda\vdash n} c_{\pi,\lambda,k} (q) \footnotesize{\begin{bmatrix}\alpha+k\\n\end{bmatrix}_q} s_\lambda[X],
\]
then 
$$c_{\pi,\lambda,k}(q) = q^{\area(\pi)+\binom{n}{2}}c_{\pi,\lambda',n-1-k}(q^{-1}).$$
\end{prop}

\begin{proof}
\cite[Proposition 4.1.4]{BHMPS} states that for an $n$-Dyck path $\pi$, 
 \[
    \omega(\LLT_\pi [X;q]) = q^{\area(\pi)}\LLT_\pi[X;q^{-1}].
 \]
Taking the plethystic substitution $X\mapsto (q^\alpha-1)X$, we get
\begin{align*}
\omega\LLT_{\pi}[(q^{\alpha}-1)X;q] = q^{\area(\pi)}\LLT_{\pi}[((q^{-1})^{(-\alpha)}-1)X;q^{-1}].
\end{align*}
By applying Proposition~\ref{prop: plethystic relation between X and LLT Carlsson Mellit} and dividing both sides by $(q-1)^n$, we have
\begin{align*}
 \omega \chi_{\pi}^{(\alpha)}[X;q] 
 = \frac{ q^{\area(\pi)}\LLT_{\pi}[((q^{-1})^{(-\alpha)}-1)X;q^{-1}]}{(-q)^n(q^{-1}-1)^n}
  &= (-1)^n q^{\area(\pi)-n}\chi_\pi^{(-\alpha)}[X;q^{-1}].
\end{align*}
In the right-hand side of the last equation, applying 
\[
    \footnotesize{\begin{bmatrix}-\alpha+k\\n\end{bmatrix}_q} = (-1)^n q^{\binom{n+1}{2}} \footnotesize{\begin{bmatrix}\alpha+n-1-k\\n\end{bmatrix}_q},
\]
yields the conclusion.

\end{proof}

\section{Applications to $q$-rook theory}
\label{Sec:qHitpoly}

For an $n$-Dyck path $\pi$, let $c_i(\pi) = i-1-b_i(\pi)$, where $b_i(\pi)$ is the number of cells below $\pi$ and strictly above the diagonal $y=x$, 
in the $i$th column, as defined in Proposition \ref{RProd}.  Note that $c_i(\pi)$ is the number of squares   
above $\pi$ in $[n]\times [n]$, in the $i$th column from the right, so if we rotate the $[n]\times [n]$-grid $180$ degrees counterclockwise, the squares originally above $\pi$ become the 
Ferrers board $B_{\pi} =B(c_1(\pi),c_2(\pi),\ldots ,c_n(\pi))$.  For example, for $\pi =(2,4,4,6,6,7,7)$, $\pi$ and the relations of the poset $P_{\pi}$ are given on the left in Figure \ref{Rotate}.  On the right we have the board $B_{\pi}=B(c_1(\pi),c_2(\pi),\ldots ,c_n(\pi))=(0,0,1,1,3,3,5)$, with the squares of $B_{\pi}$ labelled according to the Cartesian $(\text{column},\text{row})$ coordinates.  By Proposition \ref{RProd} we will have a solution to the Garsia-Remmel $q$-hit problem if we can find a combinatorial interpretation for the coefficients
in the expansion of $\chi _{\pi}[Q_{\alpha};q]$ into the  
$ \footnotesize{\begin{bmatrix}\alpha+k\\n\end{bmatrix}_q}$ basis.

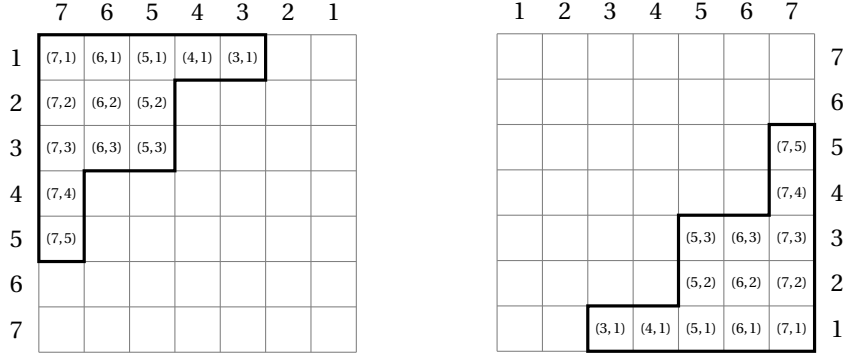
\begin{figure}[!ht]
\begin{tikzpicture}[scale=.6]
 \foreach \i in {0,...,7} {
 \draw[color=gray] (0,\i)--(7,\i);
  \draw[color=gray] (\i,0)--(\i,7);
  }
  \foreach \i in {1,...,7}{
\node () at (-.5, 7-\i+.5) {$\i$};
\node () at (7-\i+.5, 7.5) {$\i$};
}
 \draw [very thick] (0,2)--(1,2)--(1,4)--(3,4)--(3,6)--(5,6)--(5,7)--(0,7)--(0,2)--cycle;
 \node () at (.5, 2.5) {\tiny$(7,5)$};
  \node () at (.5, 3.5) {\tiny$(7,4)$};
   \node () at (.5, 4.5) {\tiny$(7,3)$};
    \node () at (.5, 5.5) {\tiny$(7,2)$};
     \node () at (.5, 6.5) {\tiny$(7,1)$};
   \node () at (1.5, 4.5) {\tiny$(6,3)$};
    \node () at (1.5, 5.5) {\tiny$(6,2)$};
     \node () at (1.5, 6.5) {\tiny$(6,1)$};  
   \node () at (2.5, 4.5) {\tiny$(5,3)$};
    \node () at (2.5, 5.5) {\tiny$(5,2)$};
     \node () at (2.5, 6.5) {\tiny$(5,1)$};
        \node () at (3.5, 6.5) {\tiny$(4,1)$};
            \node () at (4.5, 6.5) {\tiny$(3,1)$};
\end{tikzpicture}
\qquad\qquad\quad
\begin{tikzpicture}[scale=.6]
 \foreach \i in {0,...,7} {
 \draw[color=gray] (0,\i)--(7,\i);
  \draw[color=gray] (\i,0)--(\i,7);
  }
  \foreach \i in {1,...,7}{
\node () at (7.5,\i-.5) {$\i$};
\node () at (\i-.5, 7.5) {$\i$};
}
\draw [very thick] (2,0)--(2,1)--(4,1)--(4,3)--(6,3)--(6,5)--(7,5)--(7,0)--(2,0)--cycle;
       \node () at (2.5, .5) {\tiny$(3,1)$};
            \node () at (3.5, .5) {\tiny$(4,1)$};
                 \node () at (4.5, .5) {\tiny$(5,1)$};
                      \node () at (5.5, .5) {\tiny$(6,1)$};
                           \node () at (6.5, .5) {\tiny$(7,1)$};
                            \node () at (4.5, 1.5) {\tiny$(5,2)$};
                      \node () at (5.5, 1.5) {\tiny$(6,2)$};
                           \node () at (6.5, 1.5) {\tiny$(7,2)$};
                            \node () at (4.5, 2.5) {\tiny$(5,3)$};
                      \node () at (5.5, 2.5) {\tiny$(6,3)$};
                           \node () at (6.5, 2.5) {\tiny$(7,3)$};
                            \node () at (6.5, 3.5) {\tiny$(7,4)$};
                             \node () at (6.5, 4.5) {\tiny$(7,5)$};
  \end{tikzpicture}
  \caption{On the right, the Ferrers board $B_{\pi}=B(c_1(\pi),\ldots ,c_n(\pi)) = B(0,0,1,1,3,3,5)$ for $\pi =(2,4,4,6,6,7,7)$, with the relations in the poset $P_{\pi}$ appearing above $\pi$ on the left}
\label{Rotate}
\end{figure}

For $\alpha \in \mathbb N$ and $D\subseteq [n-1]$, a basic identity involving the Gessel's fundamental quasisymmetric function $F_{D}$ is
\cite[Ch. 7]{Stanley2} 
\begin{align}
\label{Fidentity}
F_{D}[Q_\alpha]=  q^{n|D| -\sum_{i \in D} i} \footnotesize{\begin{bmatrix}\alpha+n-1-|D| \\n\end{bmatrix}_q}.
\end{align}
The result of Shareshian and Wachs \cite{ShWa16} can be phrased as
\begin{align}
\label{Fexp}
\chi _{\pi}[X;q] = \sum_{\sigma \in \mathfrak{S}_n} 
q^{ \inv_{\pi} (\sigma ^{-1}) } F_{ \text{PDes}^c_\pi(\sigma)}[X],
\end{align}
where $\text{PDes}_\pi(\sigma)$ is the set of $i$ for which $\sigma_i > \sigma_{i+1}$ in $P_{\pi}$, and $\text{PDes}^c_\pi(\sigma)$ is its complement
 in $[n-1]$.
Setting $X=Q_{\alpha}$ in (\ref{Fexp}) and using (\ref{Fidentity}) we get
\begin{align}
\label{QFexp}
\chi _{\pi}[Q_{\alpha};q] = \sum_{\sigma \in \mathfrak{S}_n} 
q^{\stat_{B_{\pi}}(\sigma)}
 \footnotesize{\begin{bmatrix}\alpha+|PDes_\pi(\sigma)| \\n\end{bmatrix}_q},
 \end{align}
 where 
 \begin{align}
 \label{hitstat}
 \stat_{B_{\pi}}(\sigma) = \inv_{\pi}(\sigma ^{-1}) + n(n-1)-n\left|\text{PDes}_\pi(\sigma)\right| -\sum_{i: \sigma_i \notin \text{PDes}_\pi(\sigma) }i.
 \end{align}
\noindent
 It is worth noting that \eqref{QFexp} can also be derived from Theorem~\ref{thm: alpha+k choose n basis monomial expansion} and Proposition~\ref{prop: q-chromatic}. 

 There is a standard way \cite{Riordan} to identify a permutation $\sigma \in \mathfrak{S}_n$ having $\text{PDes}_\pi(\sigma)=k$ with a placement $L_{\sigma}$ of $n$ rooks on 
  $[n]\times [n]$ with exactly $k$ rooks on $B_{\pi}$.  Start with $\sigma _1 \cdots \sigma _n$ and create a permutation $\beta (\sigma)$ by viewing the
  right-to-left minima of $\sigma$ (i.e. the values of $i$ for which $\sigma_i < \sigma_j$ for $i<j$) as defining the rightmost element in the cycles in $\beta (\sigma)$.  For example,
  if $\sigma = 52416837$, then $\beta (\sigma) = (5241)(6843)(7)$.  Now let $L_{\sigma}$ be the placement of rooks on all squares $(i,j)$ where $i$ immediately
  follows $j$ in a cycle of $\beta(\sigma)$.  For this case, we get rooks on squares in  $(\text{column},\text{row})$-coordinates
  \begin{align}
  (5,2),(2,4),(4,1),(1,5),(6,8),(8,4),(4,3),(3,8),(7,7).
  \end{align}
 Since all of $P_{\pi}$ are contained below the diagonal $y=x$, $P_{\pi}$-descents of $\sigma$ are in bijection with rooks in $L_{\sigma}$ on squares of $B_{\pi}$.

\begin{exam}
    In this example we show how to compute the statistic 
    $\text{stat}_{B_{\pi}}(\sigma)$ for the rook placement $L$ on the left in Figure \ref{fig:rooksinv}, where the red dots are rooks.   Here
    $B_{\pi} = B(0,0,1,2,2,3,5,6)$, where $\pi$ is the path forming the boundary of the Ferrers board.  The placement $L$ corresponds under the above map to $\sigma = 52416837$, which has $P$-descents at $1,3,6$.  To compute $\text{inv}_{\pi}(\sigma ^{-1})$, we place the elements of $\sigma ^{-1}$ along the diagonal, as on the right in the figure.  The blue dots correspond to inversion pairs,  so
    $\text{inv}_{\pi}(\sigma ^{-1}) = 5$.  Using this, from (\ref{hitstat}) we have
    \begin{align*}
    \text{stat}_{B_{\pi}}(\sigma) = 5+8\cdot 7 -8\cdot 3-(2+4+5+7) = 19.
    \end{align*}

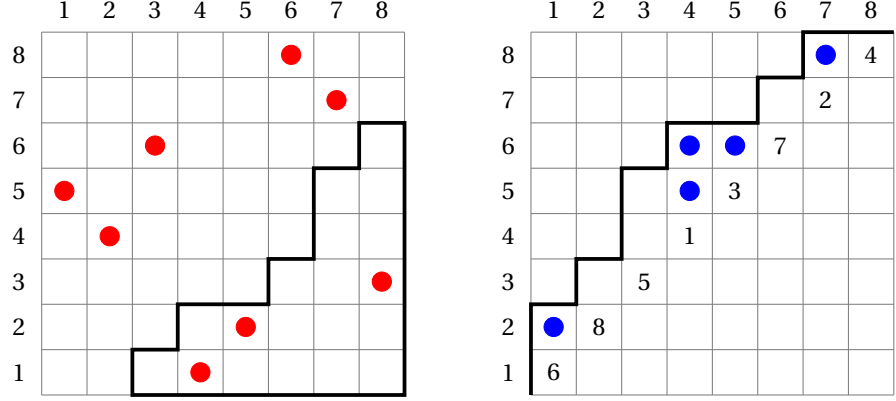
\begin{figure}[!ht]
\begin{tikzpicture}[scale=.6]
 \foreach \i in {0,...,8} {
 \draw[color=gray] (0,\i)--(8,\i);
  \draw[color=gray] (\i,0)--(\i,8);
  }
  \foreach \i in {1,...,8}{
\node () at (-.5, \i-.5) {$\i$};
\node () at (\i-.5, 8.5) {$\i$};
}
 \draw [line width=0.5mm] (2,0)--(2,1)--(3,1)--(3,2)--(5,2)--(5,3)--(6,3)--(6,5)--(7,5)--(7,6)--(8,6)--(8,0)--(2,0)--cycle;
 \filldraw[color=red] (.5, 4.5) circle (6pt)
                     (1.5, 3.5) circle (6pt)
                     (2.5, 5.5) circle (6pt)
                     (3.5, .5) circle (6pt)
                     (4.5, 1.5) circle (6pt)
                     (5.5, 7.5) circle (6pt)
                     (6.5, 6.5) circle (6pt)
                     (7.5, 2.5) circle (6pt);
\end{tikzpicture}
\qquad\quad
\begin{tikzpicture}[scale=.6]
 \foreach \i in {0,...,8} {
 \draw[color=gray] (0,\i)--(8,\i);
  \draw[color=gray] (\i,0)--(\i,8);
  }
  \foreach \i in {1,...,8}{
\node () at (-.5, \i-.5) {$\i$};
\node () at (\i-.5, 8.5) {$\i$};
}
 \draw [line width=0.5mm] (0,0)--(0,2)--(1,2)--(1,3)--(2,3)--(2,5)--(3,5)--(3,6)--(5,6)--(5,7)--(6,7)--(6,8)--(8,8);
 \filldraw[color=blue] (.5, 1.5) circle (6pt)
                     (3.5, 4.5) circle (6pt)
                     (3.5, 5.5) circle (6pt)
                     (4.5, 5.5) circle (6pt)
                     (6.5, 7.5) circle (6pt);
\node () at (.5, .5) {$6$};
\node () at (1.5, 1.5) {$8$};
\node () at (2.5, 2.5) {$5$};
\node () at (3.5, 3.5) {$1$};
\node () at (4.5, 4.5) {$3$};
\node () at (5.5, 5.5) {$7$};
\node () at (6.5, 6.5) {$2$};
\node () at (7.5, 7.5) {$4$};
\end{tikzpicture}
 \caption{A placement of rooks $L_{\sigma}$ on the left. Here $\sigma=52416837$, with $\text{inv}_{\pi}(\sigma ^{-1})=5$, corresponding to the blue dots on the right.}\label{fig:rooksinv}
\end{figure}
  
\end{exam}

Note that our argument expressing 
values of the $q$-hit polynomial of $B_{\pi}$ in terms of $\text{inv}_{\pi}(\sigma ^{-1})$ depends only 
(\ref{Fidentity}) and (\ref{Fexp}), and any way of replacing 
$\text{inv}(\sigma ^{-1})$ in (\ref{Fexp}) by a different statistic which still makes the equation true would also give a solution to the Garsia-Remmel $q$-hit problem for the board $B_{\pi}$.  Thus we have the following.
\begin{thm}
Let $B_{\pi}$ be the Ferrers board corresponding to an n-Dyck path $\pi$. Assume that  
\begin{align}
\label{GenFexp}
\chi _{\pi}[X;q] = \sum_{\sigma \in \mathfrak{S}_n} 
q^{ \text{qstat}_{\pi} (\sigma) } F_{ \text{PDes}^c_\pi(\sigma)}[X],
\end{align}
where $\text{qstat}_{\pi}(\sigma)$ is some function from $\mathfrak{S}_n$ to $\mathbb N$ depending on $\pi$.  
Then for $0\le k \le n$,
\begin{align}
    h_k(B_{\pi};q) = 
    \sum_{\sigma \in \mathfrak{S}_n \atop 
    \text{$\beta (\sigma)$ has $k$ rooks on $B_{\pi}$} }
q^{\text{hitstat}_{\pi}(\sigma)},
\end{align}
where
\begin{align}
\label{story}
\text{hitstat}_{\pi}(\sigma) =
    \text{qstat}_{\pi}(\sigma) + n(n-1)-n\left|\text{PDes}_\pi(\sigma)\right| -\sum_{i: \sigma_i \notin \text{PDes}_\pi(\sigma) }i.
\end{align}
\end{thm}

\begin{cor}
    The solution to the Garsia-Remmel $q$-hit problem for all Ferrers boards $B_{\pi}$ contained below the main diagonal extends easily
    to a solution that is valid for all Ferrers boards.
\end{cor}

\begin{proof}
 Equation \eqref{QFexp} gives a formula for the $h_k(B_{\pi};q)$ for all
 $k$ for any $n$-Dyck path $\pi$, but such $B_{\pi}$ are all contained below the diagonal $y=x$.   Now given a board 
 $B(m_1,m_2,\ldots ,m_n)$ which is
 not contained below the diagonal,  the nonincreasing rearrangement of the numbers $m_i-i+1$ must begin with a positive integer, say $J$, and correspond to a
 Ferrers board, call it $M(k_1=J,k_2,k_3,\ldots ,k_n)$.    Now all the Ferrers boards  $T$
 corresponding to rearrangements 
 of the $m_i-i+1$  have the same $q$-rook and $q$-hit numbers, and must satisfy $h_{j}(T;q)=0$ for $j<J$, since any placement of $n$ rooks has at least $J$ rooks on $T$.
 Due to this observation, we have 
 \begin{align}
 \label{subtract}
 \prod_{i=1}^n [\alpha + k_i-i+1] &= \sum_{k=0}^n \footnotesize{\begin{bmatrix}\alpha+k\\n\end{bmatrix}_q}h_{k}(T;q) \\
 \label{subtract2}
&=\sum_{k=J}^n \footnotesize{\begin{bmatrix}\alpha+k\\n\end{bmatrix}_q}h_{k}(T;q).
 \end{align}
If we replace $\alpha$ by $\alpha - J$ on both sides of (\ref{subtract}) and (\ref{subtract2}), we see that $h_k(T;q)=h_{k-J}(R)$ for the Ferrers board $R$
obtained by shifting all the columns of $T$ down by $J$.  Since $R$ 
is contained below the diagonal $y=x$, the formula (\ref{QFexp}) for the $q$-hit numbers of boards contained below the diagonal induces one for all boards contained in $[n]\times [n]$. 

\end{proof}

\section{Final remarks}\label{Sec: further remarks}

\subsection{Schur expansion of unicellular LLT polynomials}\label{subsec: Schur expansion for unicellular LLT}
\begin{prop}\label{prop: Schur coefficient of LLT}
    For an $n$-Dyck path $\pi$, let $c^k_{\pi,\lambda}(q)$ and $d^k_{\pi,\lambda}(q)$ be the Schur coefficients of the $\alpha$-chromatic symmetric function into the bases $\left\{\footnotesize{\begin{bmatrix}\alpha+k\\n\end{bmatrix}_q}\right\}_{0\le k \le n-1}$ and $\left\{[\alpha]_q^{\underline{k}}\right\}_{1\le k \le n}$, respectively, i.e.
    \begin{equation}\label{eq: expansion of alpha chromsym}
        \chi_\pi^{(\alpha)}[X;q] = \sum_{\lambda\vdash n} \sum_{k=0}^{n-1} c^k_{\pi,\lambda}(q)\footnotesize{\begin{bmatrix}\alpha+k\\n\end{bmatrix}_q} s_\lambda[X] = \sum_{\lambda\vdash n} \sum_{k=1}^{n} d^k_{\pi,\lambda}(q)[\alpha]_q^{\underline{k}}s_\lambda[X].
    \end{equation}
    Then we have
    \begin{equation}\label{eq: expansion of LLT}
        \LLT_\pi[X;q] = \sum_{\lambda\vdash n} \dfrac{\sum_{k=0}^{n-1} c^k_{\pi,\lambda}(q)}{[n]_q!}s_\lambda[X]=\sum_{\lambda\vdash n} q^{-\binom{n}{2}} d^n_{\pi,\lambda}(q)s_\lambda[X].
    \end{equation}
\end{prop}
\begin{proof}
    By dividing both sides of \eqref{eq: expansion of alpha chromsym} by $(q-1)^n$, and by taking $\alpha\rightarrow\infty$, we get \eqref{eq: expansion of LLT}.
\end{proof}
As one can see, the Schur coefficients of the $\alpha$-chromatic symmetric function expanded in the basis $\left\{\footnotesize{\begin{bmatrix}\alpha+k\\n\end{bmatrix}_q}\right\}_{0\le k \le n-1}$ yield the Schur coefficients for the LLT polynomials, multiplied by $[n]_q!$. Theorem~\ref{thm: alpha+k choose n basis monomial expansion} can be viewed as a step in this direction to find a combinatorial interpretation of Schur coefficients of the LLT polynomials.

On the other hand, when $\alpha=1$, we have $[\alpha]^{\underline{k}}_q=0$ for all $k$ except when $k=1$, and the $\alpha$-chromatic symmetric function reduces to the usual chromatic symmetric function. Consequently, in the expansion of the $\alpha$-chromatic symmetric function $\chi^{(\alpha)}_\pi[X;q]$ in the $\left\{[\alpha]_q^{\underline{k}}\right\}$ basis, the coefficient of $[\alpha]_q^{\underline{1}}$ corresponds to the chromatic symmetric function $\chi_\pi[X;q]$. Building on this observation, Proposition~\ref{prop: Schur coefficient of LLT} suggests that the coefficients of the $\alpha$-chromatic symmetric functions, when expressed in the $q$-falling factorial basis, interpolate between the chromatic symmetric functions and the unicellular LLT polynomials.

\subsection{Geometric interpretations}\label{subsec: geometric interpretation}

For an $n$-Dyck path $\pi$, let $\Hess(\pi)$ be the (regular semisimple) Hessenberg variety associated to $\pi$ (or the Hessenberg function $h(\pi)$). Tymoczko's dot action \cite{Tym08} on torus equivariant cohomology induces an $\mathfrak{S}_n$ action on the cohomology $H^*(\Hess(\pi))$. Brosnan--Chow \cite{BC18}, and Guay-Paquet \cite{Gua13} independently proved the conjecture of Shareshian--Wachs \cite{ShWa16} that this representation coincides with the chromatic symmetric function:
\[
    \omega \chi_\pi[X;q] = \Frob(H^*(\Hess(\pi));q),
\]
where $\Frob$ is the (graded) Frobenius characteristic map. 

Since the $\alpha$-chromatic symmetric functions enjoy various Schur positivity phenomena and symmetry relations, it is tempting to seek a geometric model behind it. For fixed $\alpha\in\mathbb{Z}_{>0}$, it is quite direct to obtain the following geometric interpretation for the $\alpha$-chromatic symmetric function.

\begin{prop}
    Let $\pi$ be an $n$-Dyck path and $\alpha\in \mathbb{Z}_{>0}$ be a fixed positive integer. Let $\mathbb{P}^{\alpha-1}$ be the complex projective space of dimension $\alpha-1$. Then we have
    \[
        \omega \chi^{(\alpha)}_\pi[X;q] = \Frob(H^*(\Hess(\pi));q) \otimes \Frob \left( (H^*(\mathbb{P}^{\alpha-1})^n);q\right),
    \] 
    where $\otimes$ is the Kronecker product for symmetric functions. Here, we consider an $\mathfrak{S}_n$ action on $H^*(\Hess(\pi))$ given by Tymockzo's dot action and on $H^*((\mathbb{P}^{\alpha-1})^n)$ induced from the $\mathfrak{S}_n$ action on $(\mathbb{P}^{\alpha-1})^n$ given by permuting coordinates.
\end{prop}
\begin{proof}
Each factor $\mathbb{P}^{\alpha-1}$ has cohomology ring
\[
H^*(\mathbb{P}^{\alpha-1}) \cong \mathbb{Z}[x]/\langle x^\alpha \rangle.
\]
Since the cohomology of a product is given by the tensor product of the cohomologies of the factors, we have
\[
H^*\left((\mathbb{P}^{\alpha-1})^n\right) \cong \mathbb{Z}[x_1,\dots,x_n]/\langle x_1^\alpha,\dots,x_n^\alpha \rangle,
\]
where $x_i$ denotes the pullback of the first Chern class of the hyperplane bundle from the $i$-th factor. The symmetric group $\mathfrak{S}_n$ acts on $(\mathbb{P}^{\alpha-1})^n$ by permuting the factors, which induces an action on the cohomology ring by permuting the variables $x_1,\dots,x_n$. In particular, each graded piece
\(
H^{2k}\left((\mathbb{P}^{\alpha-1})^n\right)
\)
inherits an $\mathfrak{S}_n$-representation structure, with basis
\[
\{x_1^{k_1} \cdots x_n^{k_n} : k_1 + \cdots + k_n = k,\ 0 \le k_i < \alpha \}.
\]
It follows that the Frobenius characteristic of the cohomology ring of the product of projective spaces is given by
\[
\Frob\left(H^*((\mathbb{P}^{\alpha-1})^n); q\right) = \sum_{\lambda} \sum_{\beta \in \mathfrak{S}_n \cdot \lambda} q^{\wt(\beta)} h_\lambda[X],
\]
where the sum is over partitions $\lambda$ with all parts less than $\alpha$, and $\wt(\beta) = \sum_{i=1}^n (i-1)\beta_i$.

By the Cauchy identity, we have
\[
h_n[Q_\alpha X] = \sum_{\lambda} m_\lambda[Q_\alpha] h_\lambda[X] = \sum_{\lambda} \sum_{\beta \in \mathfrak{S}_n \cdot \lambda} q^{\wt(\beta)} h_\lambda[X],
\]
and hence
\[
\Frob\left(H^*((\mathbb{P}^{\alpha-1})^n); q\right) = h_n[Q_\alpha X].
\]
Furthermore, for any symmetric function $f$, the Kronecker product (corresponding to the inner tensor product) of $f$ with $h_n[Q_\alpha X]$ simplifies as
\[
f[X] \otimes h_n[Q_\alpha X] = f[Q_\alpha X].
\]
Therefore, we obtain
\[
\Frob\left(H^*(\Hess(\pi)); q\right) \otimes \Frob\left(H^*((\mathbb{P}^{\alpha-1})^n); q\right)
= h_n[Q_\alpha X] \otimes \omega \chi_\pi[X;q] 
= \omega \chi_\pi[Q_\alpha X; q] 
= \omega \chi_\pi^{(\alpha)}[X; q],
\]
which completes the proof.
\end{proof}

If we let $\alpha$ be general (as a parameter), it is unclear how to interpret the $\alpha$-chromatic symmetric functions in terms of geometry. This approach might lead us to the desired Schur positivity in Conjecture~\ref{conj: falling factorial Schur positivity}.

\section{Acknowledgment}
The authors are grateful to Donggun Lee, Anna Pun, Brendon Rhoades, Alexander Vetter, and Jennifer Wang for the helpful discussion. We thank the anonymous referee for valuable suggestions that improved the paper.

\appendix
\section{Another proof of Theorem~\ref{thm: falling factorial basis} using the modular law}\label{Sec: appendix}

\subsection{Proof using modular law}
The \emph{modular law} is a linear relation among functions on Dyck paths. 
\begin{defn}\label{def: modular law} Let $\mathcal{A}$ be a $\mathbb{Q}$-algebra. 
 We say a function $f:\Dyck_n\rightarrow \mathcal{A}$ satisfies the \emph{modular law} if 
 \[
    2f(\pi_1) = f(\pi_0)+f(\pi_2)
 \]
when either one of the following conditions hold.
\begin{itemize}
    \item There exists $i\in[n-1]$ such that $\pi_1(i-1)<\pi_1(i)<\pi_1(i+1)$ and $\pi_1(\pi_1(i))=\pi_1(\pi_1(i)+1)$ or $\pi_1(i)=n$. Moreover, $\pi_0$ and $\pi_2$ satisfy $\pi_k(j) := \pi_1(j)$ for every $j\neq i$ and $k\in\{0,2\}$, while $\pi_k(i) =\pi_i(i)-1+k$.
    \item There exists \( i \in [n-1] \) such that \( \pi_1(i+1) = \pi_1(i) + 1 \) and \( \pi_1^{-1}(i) = \emptyset \). Moreover, \( \pi_0 \) and \( \pi_2 \) satisfy \( \pi_k(j) = \pi_1(j) \) for all \( j \neq i, i+1 \) and for \( k = 0, 2 \), while \( \pi_0(i) = \pi_0(i+1) = \pi_1(i) \) and \( \pi_2(i) = \pi_2(i+1) = \pi_1(i+1) \).

\end{itemize}
\end{defn}
Originally, this relation was discovered by Guay-Paquet \cite{Gua13} for chromatic symmetric functions for $(3+1)$-free graphs, but we restrict our attention to unit interval graphs. An analogous relation for the unicellular LLT polynomials was found by Lee \cite{Lee21}. Using the modular law, Abreu and Nigro  \cite{AN21} characterize the chromatic symmetric function as follows.
\begin{thm} Let $\mathcal{A}$ be a $\mathbb{Q}$-algebra and let $f:\Dyck_n\rightarrow \mathcal{A}$ be a function that satisfies the modular law. Then $f$ is determined by its values $f(K_{n_1}\cup K_{n_2}\cup\cdots\cup K_{n_m})$ at the disjoint ordered union of complete graphs and these values are independent of the order in which the union is taken.
\end{thm}
We now give an inductive proof of Theorem~\ref{thm: falling factorial basis} using this characterization.
\begin{proof}[Proof of Theorem~\ref{thm: falling factorial basis}]
Recall the expansion \eqref{Eq: falling factorial expansion}
\[
\chi_{\pi}^{(\alpha)}[X;1] = \sum_{k=1}^n (\alpha)^{\underline{k}}\sum_{S\in \SSS(n,k)} \chi_{\beta(\pi,S)}[X;1].
\]

Since both sides are symmetric under taking transposes of Dyck paths, it suffices to verify that both satisfy the modular law when the triple \( (\pi_0, \pi_1, \pi_2) \) is of the first type in Definition~\ref{def: modular law}.

The left-hand side of \eqref{Eq: falling factorial expansion} clearly satisfies the modular law. Now, we proceed to demonstrate that the right-hand side of \eqref{Eq: falling factorial expansion} also follows the modular law. Consider three Dyck paths $\pi_0$, $\pi_1$, and $\pi_2$ as described in Definition~\ref{def: modular law}, where the only difference is at the $i$-th column, and $\pi_0(i) + 1 = \pi_1(i) = \pi_2(i) - 1 = j$. 
We can categorize possible cases as follows:
\begin{case}
    When both $j$ and $j+1$ are in the same part with $i$ in the set partition $S$.
\end{case} 
\begin{case}
    When both $j$ and $j+1$ are not in the same part with $i$ in the set partition $S$.
\end{case}
\begin{case}
    When one of $j$ or $j+1$ is in the same part with $i$ in the set partition $S$, and the other is not.
\end{case} 

In the first case, it is evident that the new Dyck paths $\beta(\pi_0,S)$, $\beta(\pi_1,S)$, and $\beta(\pi_2,S)$ satisfy the conditions in Definition~\ref{def: modular law}. Therefore, by applying the modular law, we obtain
\[
\chi_{\beta(\pi_2,S)}-\chi_{\beta(\pi_1,S)} = \chi_{\beta(\pi_1,S)}-\chi_{\beta(\pi_0,S)}.
\]
In the second case, all three Dyck paths $\beta(\pi_0,S)$, $\beta(\pi_1,S)$, and $\beta(\pi_2,S)$ are identical. This implies $\chi_{\beta(\pi_0,S)}=\chi_{\beta(\pi_1,S)}=\chi_{\beta(\pi_2,S)}$, resulting in
\[
\chi_{\beta(\pi_2,S)}-\chi_{\beta(\pi_1,S)} = \chi_{\beta(\pi_1,S)}-\chi_{\beta(\pi_0,S)} = 0.
\]
For the last case, without loss of generality, let us assume that $i$ and $j$ are in the same part of set partition $S$, and $j+1$ is not. Now, let $S'$ be the set partition obtained by exchanging $j$ and $j+1$ in $S$. Then, we have $\beta(\pi_1,S)=\beta(\pi_2,S)$ and $\beta(\pi_0,S')=\beta(\pi_1,S')$. Moreover, since $\pi_k(j)=\pi_k(j+1)$ for $k=0,1,2$, it follows that
\[
\beta(\pi_0,S)=\beta(\pi_1,S') \quad \text{ and } \quad \beta(\pi_1,S) = \beta(\pi_2,S').
\]
This leads to the equation
\[
    \chi_{\beta(\pi_2,S)}-\chi_{\beta(\pi_1,S)} + \chi_{\beta(\pi_2,S')}-\chi_{\beta(\pi_1,S')} = \chi_{\beta(\pi_1,S)}-\chi_{\beta(\pi_0,S)} + \chi_{\beta(\pi_1,S')}-\chi_{\beta(\pi_0,S')}.
\]
Considering all the possible cases, we can conclude that the right-hand side of \eqref{Eq: falling factorial expansion} satisfies the modular law.

Let us establish the initial case: disjoint ordered union of complete graphs. For a composition $\lambda$, let $P(\lambda)$ be the path $[\lambda_1^{\lambda_1}, \dots,\lambda_{\ell(\lambda)}^{\lambda_{\ell(\lambda)}}]$, i.e. bounce path of $\lambda$. The path $P(\lambda)$ corresponds to the disjoint union of complete graphs. We first confirm the validity of \eqref{Eq: falling factorial expansion} for the `full' Dyck path $\pi = P(n) = [n, n, \dots, n]$. Note that $\chi_\pi[X; q] = [n]_q! e_n[X]$. This leads us to the following expression
\[
\chi^{(\alpha)}_\pi[X; q] = \chi_\pi\left[\frac{1 - q^\alpha}{1 - q}X; q\right] = [n]_q! e_n\left[\frac{1 - q^\alpha}{1 - q}X; q\right] = [n]_q! \sum_{\lambda \vdash n} m_\lambda\left[\frac{1 - q^\alpha}{1 - q}\right] e_\lambda[X],
\]
For the last equality, we used  
\[
    e_n[XY] = \sum_{\lambda} b_\lambda[X] \omega(b'_\lambda[Y]),
\]
for some basis $b_\lambda$ and its dual basis $b'_\lambda$ for $\Lambda_F$. See \cite[Sec I.4]{Mac1995} for example. By setting $q = 1$, we derive the following equation
\begin{equation}\label{eq: full Dyck path e-expansion}
\chi^{(\alpha)}_\pi[X; 1] = n! \sum_{\lambda \vdash n} \begin{pmatrix}
    \alpha \\ \ell(\lambda)
\end{pmatrix} \binom{\ell(\lambda)}{m_1(\lambda),m_2(\lambda),\dots} e_\lambda[X].
\end{equation}
Moreover, the chromatic symmetric function of the bounce path $P(\lambda)$ is $\chi_{P(\lambda)}[X;1] = \lambda! e_\lambda$. Here,  we used the notation $\lambda!=\lambda_1! \lambda_2! \cdots$. For each set partition $S$ of $n$, the associated path $\beta(\pi,S)$ is the bounce path
\[
    \beta(\pi,S) = P(\lambda(S)),
\]
where $\lambda(S)$ is the composition $(|S^{(1)}|,\dots,|S^{\ell(S)}|)$. Therefore, the right-hand side of \eqref{Eq: falling factorial expansion} can be expressed as
\[
\sum_{k=1}^n (\alpha)^{\underline{k}}\sum_{S\in \SSS (n,k)} X_{\beta(\pi,S)} = \sum_{k=1}^n (\alpha)^{\underline{k}}\sum_{S\in \SSS (n,k)} \lambda(S)! e_{\lambda(S)} = \sum_{\lambda \vdash n} (\alpha)^{\underline{\ell(\lambda)}} \begin{pmatrix}
    n \\ \lambda
\end{pmatrix} \frac{1}{\prod_{i\ge 1} m_i(\lambda)!} \lambda! e_\lambda,
\]
This expression aligns with \eqref{eq: full Dyck path e-expansion}. 

Our next task is to prove that \eqref{Eq: falling factorial expansion} holds for $\pi =\pi(n_1,\dots,n_\ell)$ for a composition $(n_1,\dots,n_\ell)$. Given that $\chi^{(\alpha)}_\pi[X;1] = \chi^{(\alpha)}_{K_{n_1}}[X;1] \cdots \chi^{(\alpha)}_{K_{n_\ell}}[X;1]$, we arrive at the following expression
\begin{align}
\chi^{(\alpha)}_\pi[X;1] &= \sum_{i=1}^\ell \sum_{1\le k_1\le n_1,\cdots,1\le k_\ell\le n_\ell} \sum_{S_i\in \SSS(n_i,k_i)} (\alpha)^{\underline{k_1}}\cdots (\alpha)^{\underline{k_\ell}}\chi_{P(\lambda(S_1,\dots,S_\ell))}[X;1]\\
&= \sum_{\lambda \vdash n} \sum_{S_1,\dots,S_\ell} (\alpha)^{\underline{\ell(S_1)}}\cdots (\alpha)^{\underline{\ell(S_\ell)}} \chi_{P(\lambda)}.\label{eq: equation a3}
\end{align}
In the right-hand side of \eqref{eq: equation a3}, the second sum is over set partitions $S_1, \dots, S_\ell$ of $n_1, \dots, n_\ell$ that result in $\lambda(S_1,\dots,S_\ell)=\lambda$. To further explain, the product of falling factorials is expanded in terms of a falling factorial basis as follows \cite[Ch 2.4]{StaEC}
\begin{equation}\label{eq: falling factorial product to sum}
(\alpha)^{\underline{k_1}}\cdots(\alpha)^{\underline{k_\ell}} = \sum_{k\ge1} c_{k,k_1,\dots,k_\ell} (\alpha)^{\underline{k}},    
\end{equation}
where $c_{k,k_1,\dots,k_\ell}$ denotes the number of set partitions of $k_1 + \cdots + k_\ell$ into $k$ parts, such that the numbers $1, 2, \dots, k_1$ are in distinct partitions, $k_1 + 1, k_1 + 2, \dots, k_1 + k_2$ are in distinct partitions, and so forth. Employing \eqref{eq: falling factorial product to sum} in \eqref{eq: equation a3} completes the proof.
\end{proof}


\bibliographystyle{alpha}  
\bibliography{alpha.bib}

\end{document}